%% file: multigoal.tex
\documentclass[12pt,reqno]{NumPDEsArticle}

\usepackage[utf8]{inputenc}
\usepackage[T1]{fontenc}
\usepackage[english]{babel}

\usepackage{csquotes}
\usepackage{enumitem}
\usepackage{amsmath,amssymb,amsthm,color}
\usepackage{hyperref}
\usepackage{enumitem}
\usepackage{ifthen}
\usepackage{nameref}
\usepackage{mathtools}
\usepackage{stmaryrd}
\usepackage{orcidlink}
\usepackage{microtype}
\usepackage{todonotes}
\usepackage{xcolor}
\usepackage[basicdelimiters]{NumPDEsMacros}

\usepackage{fancyhdr}

\newcommand{\f}{\boldsymbol{f}}
\newcommand{\g}{\boldsymbol{g}}

\let\d\relax
\newcommand{\d}[1]{\, \textup{d}#1}

\newcommand{\Ctail}{C_{\textup{tail}}}
\newcommand{\zetaa}{\zeta^{\textup{active}}}

\usepackage{graphicx}
\usepackage{subcaption}
\tikzset{meshStyle/.style={very thin, color = black!75!white, fill=none, line join=round, line
cap=round}}

\usepackage{pgfplots}
\usepackage{pgfplotstable}
\pgfplotsset{%
	compat=newest,%
	every axis/.style={scale only axis},%
	grid style={densely dotted, semithick},%
}

\usepackage{tkz-euclide}
\usetikzlibrary{shapes.geometric, calc, patterns}

\definecolor{TUBlue}{HTML}{006699}
\definecolor{TUGreen}{HTML}{007E71}
\definecolor{TUGrey}{HTML}{646363}
\definecolor{TUMagenta}{HTML}{BA4682}
\definecolor{TUYellow}{HTML}{E18922}

\title[MultiGOAFEM]{Multigoal-oriented adaptive finite element methods\\with convergence rates}

\author{Roland Becker~\orcidlink{0000-0003-2921-7855}}
\author{Maximilian Brunner~\orcidlink{0000-0003-0636-1491}}
\author{Paula Hilbert~\orcidlink{0009-0005-0105-1066}}
\author{\\ Michael Innerberger~\orcidlink{0000-0002-7001-8682}} 
\author{Dirk Praetorius~\orcidlink{0000-0002-1977-9830}}

\address{Université de Pau et des Pays de l’Adour, IPRA-LMAP, Avenue de l'Université BP 1155, 64013 PAU Cedex, France}
\email{roland.becker@univ-pau.fr}

\address{TU Wien, Institute of Analysis and Scientific Computing, Wiedner Hauptstr. 8--10/E101/4, 1040 Vienna, Austria}
\email{maximilian.brunner@asc.tuwien.ac.at}
\email{paula.hilbert@asc.tuwien.ac.at \quad {\normalfont{(corresponding author)}}}
\email{dirk.praetorius@asc.tuwien.ac.at}

\address{Janelia Research Campus, Howard Hughes Medical Institute, Ashburn, VA, USA}
\email{innerbergerm@hhmi.org}

\thanks{This research was funded in whole or in part by the Austrian Science Fund (FWF) projects
\href{https://www.fwf.ac.at/en/research-radar/10.55776/F65}{10.55776/F65} (SFB
F65 ``Taming complexity in PDE systems''),
\href{https://www.fwf.ac.at/en/research-radar/10.55776/I6802}{10.55776/I6802}
(international project I6802 ``Functional error estimates for PDEs on unbounded
domains''),
\href{https://www.fwf.ac.at/en/research-radar/10.55776/PAT3699424}{10.55776/PAT3699424}
(standalone project PAT3699424 ``Optimal robust solvers for reliable and efficient AFEMs''),
and
\href{https://www.fwf.ac.at/en/research-radar/10.55776/PAT3446525}{10.55776/PAT3446525}
(standalone project PAT3446525 ``Adaptive Uzawa-type FEM for nonlinear PDEs'').
Additionally, Paula Hilbert is supported by the Vienna School of Mathematics and
Michael Innerberger is supported by HHMI Janelia.}

\makeatother

\begin{document}


\maketitle
\thispagestyle{fancy}

\begin{abstract}
We formulate and analyze a goal-oriented adaptive finite element method for a symmetric linear elliptic partial differential equation (PDE) that can simultaneously deal with multiple linear goal functionals.
In each step of the algorithm, only two linear finite element systems have to be solved.
Moreover, all finite element solutions are computed with respect to the same discrete space, while the underlying triangulations are adapted to resolve all inherent singularities simultaneously.
Unlike available results for such a setting in the literature, we give a thorough convergence analysis and verify that our algorithm guarantees, in an appropriate sense, even optimal convergence rates. 
Numerical experiments underline the derived theoretical results.
\end{abstract}


\input{01_introduction.tex}
\input{02_algorithmN.tex}
\input{03_convergenceN.tex}

\input{04_ratesN.tex}

\input{05_numerics.tex}
\input{06_conclusion.tex}


\renewcommand*{\bibfont}{\footnotesize} 

\printbibliography

\end{document}

%% file: 01_introduction.tex
\section{Introduction}

\subsection{Overview and state of the art}
\textsl{A~posteriori} error estimation and adaptive mesh refinement are key tools for the efficient numerical solution of partial differential equations (PDEs). In particular, adaptive finite element methods (AFEMs) are able to resolve singularities and other localized features of the PDE solution with optimal convergence rates—both with respect to the number of degrees of freedom and the overall computational cost (or runtime); see, e.g.,~~\cite{s2007, ckns2008, cg2012, axioms, ghps2021,bfmps2025}. Typically, an AFEM algorithm consists of the following modules, which are applied iteratively:
\begin{center}
  \hspace{-0.8em}
         \begin{tikzpicture}
         	\tikzstyle{afemnode} = [draw, ultra thick, color=TUBlue, text=black, minimum width=5em, rounded corners]
         	\tikzstyle{afemarrow} = [ultra thick, color=TUBlue, -stealth, rounded corners]
         	\tikzstyle{dummynode} = [draw=none]
         	
         	\node[afemnode] (S) at (0,0) {\texttt{SOLVE}};
          \node[left=2em of S] (S-) {};
          \node[afemnode,right=2em of S] (E) {\texttt{ESTIMATE}};
         	\node[afemnode, right=2em of E] (M) {\texttt{MARK}};
         	\node[afemnode, right=2em of M] (R) {\texttt{REFINE}};
          \node[afemnode, right=2em of R] (S2) {\texttt{SOLVE}};
          \node[right=2em of S2] (S2+) {};
         	
         	\draw[afemarrow] (S) -- (E);
          \draw[afemarrow] (E) -- (M); 
         	\draw[afemarrow] (M) -- (R);
          \draw[afemarrow] (R) -- (S2);
         	
         	\draw[afemarrow] (S-) -- (S);
          \draw[afemarrow] (S2) -- (S2+);
         \end{tikzpicture}
\end{center} 
\noindent

In applications, the efficient approximation of certain derived \emph{quantities of interest} of the PDE solution is often the main target, while the optimal approximation of the PDE solution itself is only secondary. These \emph{quantities of interest} are typically modeled by \emph{goal functionals} and computationally treated as right-hand sides in the adjoint PDE problem. An example of such goals is the mean value of the PDE solution in a certain subdomain. 
 
Goal-oriented adaptivity strives to strike an algorithmic balance between singularities arising from the primal problem associated with the approximation of the PDE solution and with those that originate from the dual (adjoint) problems. 
For the case of a single quantity of interest, the literature on goal-oriented adaptivity is fairly mature; see, e.g.,~\cite{MR1352472, MR2009692, gs2002, ms2009} for seminal contributions in the field and~\cite{bet2011, fpz2016, MR4250624, bip2020, bbimp2022, bgip2023, bbps2025} for some more recent contributions. 

More formally (and with further details provided below), let $u^\star$ denote the exact solution to the considered PDE and let $G$ denote a linear goal functional that defines the \emph{quantity of interest}, $G(u^\star)$. Goal-oriented adaptivity then seeks to control the \emph{goal error} 
$|G(u^\star) - G(u_H^\star)|$,
where $u_H^\star$ denotes the finite element (FE) approximation of $u^\star$ in some discrete space $\XX_H$. A crucial feature of goal-oriented algorithms is that, via dualization techniques (see, e.g.,~\cite{gs2002}), this goal error can be estimated and bounded by the product of the primal and dual errors. This product allows, roughly speaking, doubled convergence rates compared to the best possible approximation of the PDE solution itself; see~\cite{ms2009} for the first result on optimal convergence rates for goal-oriented AFEM for the Poisson model problem.

Motivated by multiphysics applications such as multiphase flow or electromagnetics where multiple quantities of interest arise naturally, the algorithmic treatment of multiple goal functionals has attracted growing attention in recent years. Nonetheless, due to the technical challenges, only few works address the case of multiple quantities of interest; see, e.g.,~\cite{hh2003, h2008, elw2019, elnww2020}. In this setting, there are $2 \le N \in \N$ goals denoted by $G_j$ for $1 \le j \le N$. 

A key achievement of these works is the design of algorithms that require only two problem solves per mesh level in the \texttt{SOLVE} module—namely, the primal problem and \emph{one} suitably combined multigoal dual problem. This significantly reduces the computational cost.
However, forming this combined multigoal quantity requires knowledge of the sign of each goal-error difference, 
\[
\mathrm{sgn}(G_j(u^\star) - G_j(u_H^\star)) \quad \text{ for all } j = 1, \dots, N,
\]
which is not available in practice. The work~\cite{elw2019} observes that, under a strong saturation assumption, these signs can be recovered using the computable quantities
\[
\mathrm{sgn}(G_j(u_{h}^\star) - G_j(u_H^\star)) \quad \text{ for all } j = 1, \dots, N,
\] 
for some enriched FE space $\XX_h \supseteq \XX_H$ with corresponding Galerkin solution $u_h^\star$.

This paper presents a new multigoal-oriented AFEM with $N$ goals (NGO-AFEM) that overcomes this limitation. Our approach achieves this by cycling through all dual problems with an appropriately modified \texttt{MARK} module, while still solving only two problems in the \texttt{SOLVE} module at each mesh level. This new approach also allows for a rigorous convergence analysis that culminates in the proof of optimal convergence rates. Details on the new \texttt{MARK} module are provided below.

\subsection{Model problem}
Let $\Omega \subset \R^d$ be a bounded Lipschitz domain with $d \in \N$.
Let $\boldsymbol{A} \in [L^\infty(\Omega)]^{d \times d}_{\textup{sym}}$, $f \in L^2(\Omega)$, and $\f \in [L^2(\Omega)]^d$.
We consider the symmetric linear elliptic partial differential equation
\begin{equation}\label{eq:modelProblemStrongform}
  - \div(\boldsymbol{A} \nabla u^\star) = f - \div \f
  \quad \text{in } \Omega
  \quad \text{subject to } \quad
  u^\star = 0
  \quad \text{on } \Gamma \coloneqq \partial\Omega.
\end{equation}
The weak formulation of~\eqref{eq:modelProblemStrongform} seeks $u^\star \in H_0^1(\Omega)$ that satisfies
\begin{equation}\label{eq:weak_formulation}
  a(u^\star, v) \coloneqq \int_\Omega (\boldsymbol{A} \nabla u^\star)^\top \nabla v \d{x}
  = \int_\Omega \big( fv + \f^\top \nabla v \big) \d{x} \eqqcolon F(v)
  \quad \text{ for all } v \in H_0^1(\Omega).
\end{equation}
The Lax--Milgram lemma guarantees existence and uniqueness of the weak solution $u^\star \in H^1_0(\Omega)$ to \eqref{eq:weak_formulation} given that the diffusion coefficient $\boldsymbol{A}(x) \in \R^{d \times d}_{\rm sym}$ is uniformly positive definite, i.e., there exists $\alpha \in \R$ such that
\begin{equation}
  \inf_{x \in \Omega} \inf_{\xi \in \R^d \backslash \{0\}} \frac{\xi^\top \boldsymbol{A}(x) \xi}{\xi^\top \xi}
  \ge \alpha > 0.
\end{equation}
Under this assumption, $a(\cdot,\cdot)$ is a scalar product and $\enorm{v}^{2} \coloneqq a(v, v)$ defines an equivalent energy norm on $H^1_0(\Omega)$.
However, we are not interested in the exact solution $u^\star$ itself, but only in derived quantities of interest formulated as
evaluations of linear goal functionals
\begin{equation}
  G_j(u^\star) \coloneqq \int_\Omega \big( g_j u^\star + \g_j^\top \nabla u^\star \big) \d{x}
  \quad \text{for all } j = 1, \dots, N \in \N \text{ with } 2 \le N,
\end{equation}
where $g_j \in L^2(\Omega)$ and $\g_j \in [L^2(\Omega)]^d$ are given. For each goal functional, let
$z_j^\star \in H^1_0(\Omega)$ denote the corresponding dual solution to
\begin{equation}
  a(v, z_j^\star) = G_j(v)
  \quad \text{for all } v \in H^1_0(\Omega).
\end{equation}
To approximate $G_j(u^\star)$, consider a finite-dimensional subspace $\XX_H \subset H^1_0(\Omega)$. The Lax--Milgram lemma ensures existence and uniqueness of FE solutions $u_H^\star, z_{j,H}^\star \in \XX_H$ to
\begin{equation}\label{eq:discrete}
  a(u_H^\star, v_H) = F(v_H)
  \quad \text{and} \quad
  a(v_H, z_{j,H}^\star) = G_j(v_H)
  \quad \text{for all } v_H \in \XX_H,
\end{equation}
respectively.
For each goal $G_j$, there holds the classical error estimate from~\cite{gs2002,ms2009},
\begin{equation}
  |G_j(u^\star) - G_j(u_H^\star)|
  = | a(u^\star - u_H^\star, z_j^\star - z_{j,H}^\star)| 
  \le \enorm{u^\star - u_H^\star} \, \enorm{z_j^\star - z_{j,H}^\star}.
\end{equation}
Considering multiple goals, the adaptive strategy shall drive down the multigoal error
\begin{equation}\label{eq:multigoal:error}
  \sum_{j=1}^N |G_j(u^\star) - G_j(u_H^\star)|
  \le
  \enorm{u^\star - u_H^\star} \,
  \Big(\sum_{j=1}^N \, \enorm{z_j^\star - z_{j,H}^\star}\Big)
  \longrightarrow 0 \quad \text { as } \dim \XX_H \to \infty,
\end{equation}
with optimal algebraic rate $s + t$, where $s > 0$ is the best rate for approximating $u^\star$, while $t = \min_j t_j > 0$ is the best rate for approximating all $z_j^\star$. The error terms in~\eqref{eq:multigoal:error} can be reliably estimated from above by residual-based error estimators $\eta_H(u_H^\star)$ and $\zeta_{j,H}(z_H^\star)$ for all $ 1 \le j \le N$ for both the primal and dual problems from~\eqref{eq:discrete}, respectively. Thus, the analysis boils down to the convergence of the so-called \emph{multigoal-error estimator}
\begin{equation}\label{eq:multigoal:estimator:intro}
  \Crel^{-2} \, \sum_{j=1}^N |G_j(u^\star) - G_j(u_H^\star)|
  \le 
  \eta_H(u_H^\star) \, \sum_{j=1}^N \zeta_{j,H}(z_{j,H}^\star)
  \eqqcolon 
   \Delta_H 
   \longrightarrow 0 \text { as } \dim \XX_H \to \infty,
\end{equation}
where $0 < \Crel$ is the generic reliability constant of residual-based error estimators.
One way to ensure convergence of~\eqref{eq:multigoal:estimator:intro}  would be to solve the discrete primal as well as \emph{all} discrete dual
problems for the arising meshes $\TT_H$, i.e., one would need to solve $N+1$ finite element systems per mesh.
Instead, the proposed algorithm solves, for all arising meshes $\TT_H$, only the discrete primal as well as
\emph{one} discrete dual problem.

\subsection{Main results}
On each mesh level, the proposed multigoal-oriented AFEM with $N$ goals solves for the primal solution $u_H^\star$ and one dual solution $z_{j,H}^\star$ for one $j \in \{1, \dots, N\}$, where $j$ is cycled through as the mesh is refined. The current goal considered is called \emph{active}. The key novelty of our approach is a new marking strategy in the \texttt{MARK} module. We start with determining the quasi-minimal set of elements that need to be refined to reduce either the primal or the active dual error estimator sufficiently. This is the strategy proposed in~\cite{ms2009} for a goal-oriented AFEM with a single goal for the Poisson model problem, and refined and analyzed in~\cite{fpz2016} for general second-order linear elliptic PDEs in the Lax--Milgram setting.

We evaluate whether the active dual error estimator is still large compared to the dual error estimators of all $N-1$ previously active goals. If this is the case, we call it \emph{regular marking} and proceed with the usual marking strategy from~\cite{fpz2016}. Otherwise, we constrain the number of marked elements to be no larger than in the previous step, calling it \emph{irregular marking}. 

As a first important result, we prove that this extended marking strategy still guarantees that the multigoal-error estimator $\Delta_H$ from~\eqref{eq:multigoal:estimator:intro} contracts linearly up to some multiplicative constant; see Theorem~\ref{theorem:multigoal:linearConvergence} below.

As the main ingredient for optimal convergence rates, we realize that the cardinality control in the irregular marking allows to prove a vital cardinality estimate; see Lemma~\ref{lemma:marking} below. Exploiting this result finally establishes optimal convergence rates for the proposed multigoal-oriented AFEM with $N$ goals; see Theorem~\ref{theorem:multigoal:optimalRates} below.

\subsection{Outline}
The paper is organized as follows. Section~\ref{section:algorithmN} introduces the new multigoal-oriented adaptive algorithm. In section~\ref{section:convergenceN}, we present the proof of linear convergence, while section~\ref{section:ratesN} is devoted to the proof of optimal convergence rates. Section~\ref{section:numerics} presents numerical experiments that underline our findings, before section~\ref{section:conclusion} concludes the paper.

%% file: 02_algorithmN.tex
\section{Multigoal-oriented adaptive FEM} \label{section:algorithmN}
In this section, we formulate a multigoal-oriented AFEM that can simultaneously deal with $N \in \N$ linear goal functionals $G_1, \dots, G_N \colon \XX^\star \to \R$, while only solving two discrete problems per adaptive step.

\subsection{Mesh refinement}
\label{subsection:mesh-refinement}
We introduce some notation regarding mesh refinement. Let $\TT_0$ be a given conforming initial triangulation of $\Omega$. For mesh refinement, we employ newest vertex bisection (NVB); see~\cite{s2008, kpp2013, dgs2025}. 
For each mesh $\TT_H$ and given marked elements $\MM_H \subseteq \TT_H$, we denote by $\TT_h \coloneqq \texttt{refine}(\TT_H,\MM_H)$ the coarsest mesh where all $T \in \MM_H$ have been refined, i.e., $\MM_H \subseteq \TT_H \setminus \TT_h$. Moreover, we write $\TT_h \in \T(\TT_H)$, if $\TT_h$ results from $\TT_H$ by finitely many steps of refinement and use $\T \coloneqq \T(\TT_0)$ for brevity.

Each triangulation $\TT_H \in \T$ is associated with a conforming finite-dimensional FE space $\XX_H \subset H^1_0(\Omega)$ with piecewise polynomial ansatz functions of degree at most $p\in \N$, i.e.,
\begin{equation*}
	\XX_H
	\coloneqq
	\mathcal{S}^p_0(\TT_H)
	\coloneqq
	\set{v_H \in H_0^1(\Omega)\given \forall T \in \TT_H\colon v_H|_T \text{ is a piecewise polynomial of degree} \le p}.
\end{equation*}
Obviously, \(\XX_H \subset H_0^1(\Omega)\) is a closed subspace. Since we employ  NVB, $\TT_h \in \T(\TT_H)$ implies nestedness $\XX_H \subseteq \XX_h$ of the corresponding FE spaces.

\subsection{A~posteriori error estimation and axioms of adaptivity}
We suppose additional regularity $\boldsymbol{A}|_T \in [W^{1,\infty}(T)]^{d\times d}_{\textup{sym}}$ and $\f|_T, \g|_T \in [H^1(T)]^d$ for all $T \in \TT_H$ and define the local
contributions to the error estimators on each triangle $T \in \TT_H$ by
\begin{subequations}\label{eq:estimators}
\begin{equation}\label{eq:indicators}
  \begin{aligned}
    \eta_H(T; v_H)^2
    & \coloneqq
    h_T^2 \| f + \div(\boldsymbol{A} \nabla v_H - \f) \|_{L^2(T)}^2
    +
    h_T \| \jump{(\boldsymbol{A} \nabla v_H - \f )^\top \boldsymbol{n}} \|_{L^2(\partial T \cap \Omega)}^2,
    \\
    \zeta_{j,H}(T; v_{H})^2
    & \coloneqq
    h_T^2 \| g_j + \div(\boldsymbol{A} \nabla v_{H} - \g_j) \|_{L^2(T)}^2
    +
    h_T \| \jump{(\boldsymbol{A} \nabla v_{H} - \g_j )^\top \boldsymbol{n}} \|_{L^2(\partial T \cap \Omega)}^2,
  \end{aligned}
\end{equation}
where $h_T \coloneqq |T|^{1/d}$ denotes the local mesh-size. These are the standard residual-based indicators associated to~\eqref{eq:discrete}; see, e.g.,~\cite{ao2000, v2013} for further details. 
For any subset $\UU_H \subseteq \TT_H$ and all $ v_H \in \XX_H$, we define the error estimators by
\begin{equation}
  \eta_H(\UU_H; v_H)^2
  \coloneqq
  \sum_{T \in \, \UU_H} \eta_H(T; v_H)^2
  \quad \text{and} \quad
  \zeta_{j,H}(\UU_H; v_{H})^2
  \coloneqq
  \sum_{T \in \, \UU_H} \zeta_{j,H}(T; v_{H})^2.
\end{equation}
\end{subequations}
Furthermore, we abbreviate the global error estimators by $\eta_H(v_H) \coloneqq \eta_{H}(\TT_H; v_H)$ and
$\zeta_{j,H}(v_{H}) \coloneqq \zeta_{j,H}(\TT_H; v_{H})$.
Moreover, the argument $v_H$ is omitted if the corresponding discrete solution is used, e.g., $\eta_H(\UU_H) \coloneqq \eta_H(\UU_H,u_H^{\star})$. 
With this understanding, the \textsl{a~posteriori} error estimators satisfy the following estimator properties.
\begin{lemma}[Axioms of adaptivity~\cite{axioms}]
  \label{lemma:axioms}
  Any of the error estimators $\mu_H \in \{\eta_H, \zeta_{1,H}, \ldots, \zeta_{N,H}\}$ from~\eqref{eq:estimators}
  satisfy the following: There exist $\Cstab, \Crel, \Cdrel,
  \Cmon> 0$ and $0 < \qred < 1$ such that, for any
  triangulation $\TT_H \in \T$, any refinement $\TT_h
  \in\T(\TT_H)$, any subset $\UU_H \subseteq \TT_h \cap \TT_H$, and arbitrary $v_H \in \XX_H$, $v_h \in \XX_h$, the following hold:
  \begin{enumerate}[font=\upshape, label=\textbf{\textrm{(A\arabic*)}}, ref=A\arabic*]
    \item
      \label{axiom:stability}
      \textbf{\textup{Stability:}} 
      \(\mkern36mu  \vert \mu_h (\UU_H; v_h) -
        \mu_H(\UU_H; v_H) \vert
      \le \Cstab \, \enorm{v_h - v_H};\)
    \item \label{axiom:reduction}\textbf{\textup{Reduction:}} 
     \(\mkern104mu \mu_h(\TT_h \setminus \TT_H; v_H)
        \le
      \qred \, \mu_H(\TT_H \setminus \TT_h; v_H); \)
    \item
      \label{axiom:reliability}
      \textbf{\textup{Reliability:}}
      \(\mkern136mu \enorm{u^\star - u_H^\star}
        \le
        \Crel \, \mu_H;\)      
    \item[\textbf{(A4)}]\refstepcounter{enumi}
      \label{axiom:discrete_reliability}
      \textbf{\textup{Discrete reliability:}}
      \(\mkern62mu
        \enorm{u_h^\star - u_H^\star}
        \le
      \Cdrel \, \mu_H(\TT_H \backslash \TT_h); \)
      \renewcommand{\theenumi}{QM}
    \item[\textbf{(QM)}]\refstepcounter{enumi}
      \label{axiom:qm}
      \textbf{\textup{Quasi-monotonicity:}}
      \(\mkern122mu \mu_h \le \Cmon \mu_H. \)
  \end{enumerate}
  The constant $\Crel$ depends only on the diffusion matrix $\boldsymbol{A}$, the uniform
  \(\gamma\)-shape
  regularity of
  all
  $\TT_H \in \T$, and on
  the space dimension $d$, while $\Cstab$ and
  \(\Cdrel\) additionally depend on the polynomial degree
  $p$.
  For NVB, reduction~\eqref{axiom:reduction} holds with
  $\qred \coloneqq 2^{-1/(2d)}$.
  Moreover, 
  discrete reliability~\eqref{axiom:discrete_reliability} yields reliability~\eqref{axiom:reliability} with $\Crel \le \Cdrel$ and 
  the constant in
  quasi-monotonicity~\eqref{axiom:qm}
  satisfies
  $\Cmon \le 1 + \Cstab \, \Cdrel$.
  \qed
\end{lemma}
We recall a well-known result which follows by elementary calculus from stability~\eqref{axiom:stability} and reduction~\eqref{axiom:reduction} of the residual error estimator from~\eqref{eq:estimators}.
The statement is implicitly found in~\cite{ckns2008, axioms} and explicitly formulated in~\cite[Lemma~7]{bbps2025}. 

\begin{lemma}\label{lemma:stability+reduction}
Suppose stability~\eqref{axiom:stability} and reduction~\eqref{axiom:reduction}.
Let $\TT_H \in \T$ and $\TT_h \in \T(\TT_H)$. Let $v_H \in \XX_H$ and $v_h \in \XX_h$.
Then, there holds stability 
\begin{equation}\label{eq:stability}
 \mu_h(v_h) 
 \le 
 \mu_H(v_H) 
 + 
 \Cstab \,\enorm{v_H - v_h} 
 \quad \text{ for any } \mu_H \in \{\eta_H, \zeta_{1,H}, \ldots, \zeta_{N,H}\}.
\end{equation}
Moreover, let $0 < \theta \le 1$ and suppose the D\"orfler marking criterion $\theta \, \mu_H(v_H)^2 \le \mu_H(\TT_H \backslash \TT_h; v_H)^2$ with respect to the refined elements $\TT_H \backslash \TT_h$.
Then, with $0 < \qest \coloneqq [1-(1-\qred^2)\theta]^{1/2} < 1$, there holds
\begin{equation}\label{eq:reduction}
 \mu_h(v_h) 
 \le
  \qest \, \mu_H(v_H) 
  + \Cstab \,\enorm{v_H - v_h}
  \quad 
  \text{ for any } \mu_H \in \{\eta_H, \zeta_{1,H}, \ldots, \zeta_{N,H}\}. \hfill  \qed
\end{equation}
\end{lemma}

\subsection{Multigoal-oriented adaptive algorithm}
To present the algorithm, we need the following notation: 
For a marking parameter $0 < \theta \le 1$, a mesh $\TT_H$, and a sequence $\mu_H \in \{\eta_H, \zeta_{1,H}, \ldots, \zeta_{N,H}\}$ of corresponding refinement indicators, we let
\begin{equation}
  \M[\theta, \infty; \TT_H, \mu_H]
  \coloneqq
  \set{\UU_H \subseteq \TT_H \given \theta \, \mu_H^2 \le \mu_H(\UU_H; T)^2}
\end{equation}
denote the set of all subsets $\UU_H$ of $\TT_H$ that satisfy the D\"orfler marking
criterion~\cite{doerfler1996}, and, for $1 \le \Cmark$ and $\# \UU_H^{\textup{min}} \coloneqq \min_{\UU_H \in \M[\theta,\infty; \TT_H, \mu_H]} \#\UU_H$, we define
\begin{equation}
  \M[\theta, \Cmark; \TT_H, \mu_H]
  \coloneqq
  \set{\MM_H \in \M[\theta, \infty; \TT_H, \mu_H]
  \given 
  \#\MM_H \le \Cmark \, \# \UU_H^{\textup{min}}}
\end{equation}
to be the set of all those subsets $\UU_H$ that are quasi-minimal, i.e., minimal up to the fixed factor $\Cmark$.
The following algorithm builds upon the marking strategy from~\cite{fpz2016} in step~\eqref{algorithm:mark:fpz} but one could alternatively employ the strategy from~\cite{ms2009} while still obtaining the same results.
The algorithm reads as follows.

\begin{algorithm}[multigoal-oiented adaptive FEM (NGO-AFEM)]\label{algorithm:Ngoal}
\textbf{Input:} Initial mesh $\TT_0$ of $\Omega$, goals $G_k$ for $1 \le k \le N$ with $2 \le N$, and adaptivity parameters $0 < \theta \le 1$,\, $0 < \varrho_{\textup{irr}} < 1/(N-1)$, and $1 \le \Cmark$. Define $\zetaa_{-k} \coloneqq 0$ for $1 \le k \le N$.

\noindent\textbf{\textup{Loop:}} For all $\ell = 0, 1, 2, \dots$, repeat the following steps~{\rm(i)--(iv)}:
\begin{enumerate}[label = \rm (\roman*), ref = \rm \roman*, font = \upshape]
\item\label{algorithm:solve:primal} \textup{\texttt{SOLVE} \textbf{(primal problem) \texttt{\& ESTIMATE}\textbf{.}}} 
Compute the primal discrete solution $u_\ell^\star \in \XX_\ell$ to~\eqref{eq:discrete} and the corresponding refinement indicators $\eta_\ell(T)$ from~\eqref{eq:estimators} for all $T \in \TT_\ell$, and pick $\MM_\ell^u \in \M[\theta, \Cmark; \TT_\ell, \eta_\ell]$.
\item\label{algorithm:solve:dual} \textup{\texttt{SOLVE} \textbf{(one dual problem) \texttt{\& ESTIMATE}\textbf{.}}} 
\begin{enumerate}[label=\rm(\alph*), ref=\theenumi.\alph*, font = \upshape]
\item \label{algorithm:step:even}
Define $j \coloneqq {\rm mod}(\ell,N) + 1$, i.e., we cycle through $j = 1, \dots, N$ as $\ell$ increases.
\item\label{algorithm:step:odd}
Compute one dual discrete solution $z_{j,\ell}^\star \in \XX_\ell$ to~\eqref{eq:discrete} and the corresponding refinement indicators $\zetaa_\ell(T) \coloneqq \zeta_{j,\ell}(T)$ from~\eqref{eq:estimators} for all $T \in \TT_\ell$, and pick $\MM_\ell^z \in \M[\theta, \Cmark; \TT_\ell, \zeta_{j,\ell}]$.
\end{enumerate}
\item \label{algorithm:markall} \textup{\textbf{\texttt{MARK}.}} Determine the set of marked elements $\MM_\ell \subseteq \TT_\ell$ as follows:
\begin{enumerate}[label=\rm(\alph*), ref=\theenumi.\alph*, font = \upshape]
\item\label{algorithm:mark:fpz}
Pick $\MM_\ell^{\rm min} \in \{ \MM_\ell^u, \MM_\ell^z \}$ with $\#\MM_\ell^{\rm min} = \min \{ \#\MM_\ell^u, \#\MM_\ell^z \}$ and choose $\MM_\ell^{uz} \subseteq \MM_\ell^u \cup \MM_\ell^z$ such that $\MM_\ell^{\rm min} \subseteq \MM_\ell^{uz}$ and $\#\MM_\ell^{uz} \le \Cmark \# \MM_\ell^{\rm min}$.
\item\label{algorithm:mark} 
\textup{\textbf{If}} \quad   
$ \varrho_{\textup{irr}} \max\limits_{i=1,\dots,N-1} \, \zetaa_{\ell-i} \le \zetaa_\ell$, 

\textup{\textbf{then}} 
\quad choose $\MM_\ell \coloneqq \MM_\ell^{uz}$; \hfill \emph{\textbf{(regular marking)}}

\textup{\textbf{else}}$\mkern6mu$ \quad choose $\MM_\ell \subseteq \MM_\ell^{uz}$ with $\#\MM_\ell \le \#\MM_{\ell-1}$; \hfill \emph{\textbf{(irregular marking)}}

i.e., the else-clause covers the cases in which $1 \le \ell $ and there exists $i \in \{1, \dots, N-1\}$ such that $\zetaa_\ell < \varrho_{\textup{irr}} \, \zetaa_{\ell-i}$.
\end{enumerate}
\item\label{algorithm:refine} \textup{\textbf{\texttt{REFINE}.}}
Employ newest vertex bisection to generate $\TT_{\ell+1} = \mathtt{refine}(\TT_\ell, \MM_\ell)$.
\end{enumerate}
\textbf{Output:} Sequence of adaptively generated triangulations $\TT_\ell$ and corresponding approximate goal values $G_1(u_\ell^\star), \dots, G_N(u_\ell^\star) \in \R$ for all $\ell \in \N_0$.
\end{algorithm}

In the following, we speak of \emph{regular marking} in step $\ell$ if $\MM_\ell^u \subseteq \MM_\ell$ or $\MM_\ell^z \subseteq \MM_\ell$ (and hence D\"orfler marking is applied for either the primal estimator or the active dual estimator).
Otherwise, we speak of \emph{irregular marking} in step $\ell$.

\begin{remark}
{\rm (i)} Note that for $N = 1$, Algorithm~\ref{algorithm:Ngoal} coincides with the usual GOAFEM algorithms from~\cite{ms2009, fpz2016} and no irregular marking step is present in the algorithm.

{\rm(ii)} Obviously, one could solve in step~\eqref{algorithm:solve:dual} for \emph{all discrete dual solutions} $z_{j,\ell}^\star$ for $j = 1, \dots, N$ and then determine $\MM_\ell^z$ with respect to the \emph{overall dual estimator} $\zeta_\ell(T)^2 \coloneqq \sum_{j=1}^N \zeta_{j,\ell}(T)^2$. 
This would allow to omit the additional marking step~\eqref{algorithm:mark} and to directly apply the results of~\cite{fpz2016} to see that R-linear convergence~\eqref{eq:multigoal:linearConvergence} of Theorem~\ref{theorem:multigoal:linearConvergence} and optimal convergence rates~\eqref{eq:multigoal:thm:optimal} of Theorem~\ref{theorem:multigoal:optimalRates} hold.
However, this requires to solve for $N+1$ FE solutions per adaptive step $\ell \in \N_0$. 
Instead, the proposed algorithm requires to solve only for two FE solutions per $\ell \in \N_0$, independently of the number $N$ of goal functionals. 
The important impact of the proposed algorithm is that Theorem~\ref{theorem:multigoal:linearConvergence} and Theorem~\ref{theorem:multigoal:optimalRates} transfer from the naive algorithm to the present (much more efficient) variant.

{\rm (iii)} The irregular marking strategy in step~\eqref{algorithm:mark} does not require that any elements are chosen for refinement, i.e., $\MM_\ell = \emptyset$ is admissible. In this case, the mesh remains unchanged, i.e., $\TT_{\ell+1} = \TT_\ell$. This case is also covered in our theory and practically corresponds to a more aggressive form of adaptive refinement (which, in the worst case, solves all dual problems for one refinement step); see also the experiment in subsection~\ref{experiment:nonsmooth}.

{\rm(iv)} If the marking step~\eqref{algorithm:mark} is omitted and $\MM_\ell = \MM_\ell^{uz}$ is used for all $\ell \in \N_0$, the R-linear convergence~\eqref{eq:multigoal:linearConvergence} of Theorem~\ref{theorem:multigoal:linearConvergence} remains valid, even with a simplified proof. 
However, the statement~\eqref{eq:multigoal:thm:optimal} of Theorem~\ref{theorem:multigoal:optimalRates} is weakened to 
\begin{equation*}
 \sup_{\ell \in \N_0}(\#\TT_\ell - \#\TT_0 + 1)^{s+t} \min_{j = 1, \dots, N} \eta_\ell\, \zeta_{j,\ell}
\le C_{\rm opt} \, \max_{j=1, \dots, N} \, \norm{u^\star}_{\mathbb{A}_s} \,\norm{z_j^\star}_{\mathbb{A}_t},
\end{equation*}
i.e., optimal convergence would only be guaranteed for the goal error with the \emph{minimal} estimator product, while the other products would potentially decay at suboptimal rates.
Clearly, this is not acceptable.
\end{remark}

The subsequent sections are devoted to the convergence analysis of Algorithm~\ref{algorithm:Ngoal}: We prove linear convergence in section~\ref{section:convergenceN} and optimal convergence rates in section~\ref{section:ratesN} below. 

%% file: 03_convergenceN.tex
\section{R-linear convergence}\label{section:convergenceN}
This section proves that the multigoal-oriented AFEM with $N$ goals (Algorithm~\ref{algorithm:Ngoal}) guarantees linear convergence of the multigoal-estimator product defined in~\eqref{eq:quasi-error:multigoal} below. We state R-linear convergence of the algorithm, which is the main result of this section.

\begin{theorem}[R-linear convergence of NGO-AFEM]\label{theorem:multigoal:linearConvergence}
Suppose stability~\eqref{axiom:stability}, reduction~\eqref{axiom:reduction}, reliability~\eqref{axiom:reliability}, and quasi-monotonicity~\eqref{axiom:qm}.
Let $0 < \theta \le 1$, recall $0 < \qest = [1-(1-\qred^2)\, \theta]^{1/2} < 1$ from Lemma~\ref{lemma:stability+reduction} and define the multigoal-estimator product
\begin{equation}\label{eq:quasi-error:multigoal}
 \Delta_\ell 
 = 
 \eta_{\ell}\, \sum_{i=1}^N \zeta_{i,\ell}
 \quad \text{ for all } \ell \in \N_0.
\end{equation}
Choose $0 < \varrho_{\textup{irr}}=\varrho_{\textup{irr}}[\theta] < 1$ such that $\qctr \coloneqq \qest + (N-1)\,\varrho_{\textup{irr}} < 1$.
Then, Algorithm~\ref{algorithm:Ngoal} guarantees the existence of constants $0<\Clin$ and $0 < \qlin < 1$ such that $\Delta_{\ell}$ satisfies
\begin{equation}
\label{eq:multigoal:linearConvergence}
 \Delta_{\ell+n} 
 \le 
 \Clin \,\qlin^{n} \, \Delta_\ell
 \quad \text{for all } \ell, n \in \N_0.
\end{equation}
The constants $\Clin$ and $\qlin$ depend only on $\theta$, $\varrho_{\textup{irr}}$, and $N$ as well as on the constants in~\eqref{axiom:stability}--\eqref{axiom:reliability} and~\eqref{axiom:qm}.
\end{theorem}

\begin{remark}
   {\textup{(i)}}
Note that the multigoal-estimator product $\Delta_\ell$ from~\eqref{eq:quasi-error:multigoal} is equivalent to the multigoal-estimator product $\widehat{\Delta}_\ell \coloneqq \eta_\ell \, \max_{i = 1, \ldots, N} \zeta_{i,\ell}$. 

   {\textup{(ii)}} The bound on the adaptivity parameter $\varrho_{\textup{irr}}$ can be explicitly calculated, since $\qred = 2^{-1/2d}$ is known and $\theta$ and $N$ are user-defined. Even choosing $\varrho_{\textup{irr}}$ slightly above this threshold does not affect the numerical performance of the algorithm; see the experiment in section~\ref{experiment:nonsmooth}. 
\end{remark}

The main observation to prove Theorem~\ref{theorem:multigoal:linearConvergence} is the following perturbed contraction for the multigoal-estimator product $\Delta_\ell$ from~\eqref{eq:quasi-error:multigoal}.
\begin{lemma}
Under the assumptions of Theorem~\ref{theorem:multigoal:linearConvergence}, for all $\ell \in \N$, there exists $0 \le \Rho_\ell$ such that the multigoal-estimator product $\Delta_{\ell}$ from~\eqref{eq:quasi-error:multigoal} satisfies
\begin{equation}\label{eq:lemma:contraction:multigoal}
 \Delta_{\ell+2N-1} 
 \le 
 \qctr \, \Delta_\ell + \Rho_\ell
 \quad \text{with} \quad
 \sum_{k=\ell}^\infty \Rho_{k}^2 \le \Csum \, \Delta_\ell^2.
\end{equation}
The constant $0 <\Csum$ depends only on $\Cstab$, $\Crel$, $\Cmon$, and $N$.
\end{lemma}

\begin{proof}
The proof is divided into six steps.

\textbf{Step~1 (levelwise notational preliminaries).}
Suppose we are on level $\ell \in \N_0$. 
We define the map $J \colon \N_0 \to \{1,\dots, N\}$ by $J[\ell] \coloneqq {\rm mod}(\ell,N) + 1$, i.e., $J[\ell]$ denotes the active dual problem in step $\ell$ of Algorithm~\ref{algorithm:Ngoal}, i.e., $\zetaa_\ell = \zeta_{J[\ell], \ell}$. Moreover, by defining the map $L \colon \N \rightarrow \{\ell-1, \dots, \ell-N+1\}$ with $\zetaa_{L[\ell]} = \max_{i=1,\dots, N-1} \zetaa_{\ell-i}$, we also have the largest active goal in the last $N-1$ levels at hand, which is important for irregular marking.

For the remaining proof, it is convenient to define the following \emph{perturbation term}
\begin{equation}\label{eq:def:r_ell:multigoal}
 {\sf r}_\ell 
 \coloneqq
 \eta_\ell \, \sum_{i=1}^N  \, \enorm{z_{i,\ell+1}^\star - z_{i,\ell}^\star}
 + 
 \enorm{u_{\ell+1}^\star - u_\ell^\star} \,\sum_{i=1}^N \, \zeta_{i,\ell}
\end{equation}
with the corresponding multiplicative constant $\Cper \coloneqq \Cstab + (1+\Cmon)\,\Cstab^2\,\Crel$.

\textbf{Step~2 (forward contraction for regular marking in step $\boldsymbol{\ell}$).}
If $\MM_\ell^u \subseteq \MM_\ell$, then we employ reduction~\eqref{eq:reduction} for $\eta_\ell$ and stability~\eqref{eq:stability} for $\zeta_{J[\ell],\ell}= \zetaa_\ell$. 
This verifies
\begin{equation}\label{eq:step1:convergence:multigoal}
\begin{aligned}
 \eta_{\ell+1}\, \zeta_{J[\ell],\ell+1}
 &\le 
 \big[ \qest \,\eta_\ell + \Cstab\, \enorm{u_{\ell+1}^\star - u_\ell^\star} \big]
 \big[ \zeta_{J[\ell],\ell} + \Cstab \,\enorm{z_{J[\ell],\ell+1}^\star - z_{J[\ell],\ell}^\star} \big]
 \\&
 \le  \qest \, \eta_\ell \,\zeta_{J[\ell],\ell} + \Cstab \, {\sf r_\ell} + \Cstab^2 \, \enorm{u_{\ell+1}^\star - u_\ell^\star} \,\enorm{z_{J[\ell],\ell+1}^\star - z_{J[\ell],\ell}^\star}\\
& \le
 \qest \, \eta_\ell \,\zeta_{J[\ell],\ell} + \Cper \, {\sf r}_\ell.
\end{aligned}
\end{equation}
If $\MM_\ell^z \subseteq \MM_\ell$, then we employ stability~\eqref{eq:stability} for $\eta_\ell$ and reduction~\eqref{eq:reduction} for $\zeta_{J[\ell],\ell}$. This also guarantees~\eqref{eq:step1:convergence:multigoal} in case of exchanged roles.

\textbf{Step~3 (backward contraction for irregular marking in step $\boldsymbol{\ell}$).} 
If both inclusions $\MM_\ell^u \subseteq \MM_\ell$ and $\MM_\ell^z \subseteq \MM_\ell$ fail in step $\ell$, then Algorithm~\ref{algorithm:Ngoal}\eqref{algorithm:mark} guarantees that
\begin{equation}\label{eq:linearConvergence:irrMarking}
\zeta_{J[\ell],\ell}
= 
\zetaa_\ell 
<
\varrho_{\textup{irr}} \, \zetaa_{L[\ell]}
=
\varrho_{\textup{irr}} \, \zeta_{J[L[\ell]],L[\ell]}.
\end{equation}
Using this inequality and stability~\eqref{eq:stability} for the estimator $\eta_\ell$, we thus obtain
\begin{equation}\label{eq:step2:convergence:multigoal}
 \begin{aligned}
 \eta_\ell \, \zeta_{J[\ell],\ell}
 \mkern18mu & \stackrel{\mathclap{\eqref{eq:linearConvergence:irrMarking},~\eqref{eq:stability}}}< \mkern18mu
 \Big[ \eta_{L[\ell]} + \Cstab \sum_{k=1}^{\ell-L[\ell]}\enorm{u_{\ell+1-k}^\star - u_{\ell-k}^\star} \Big]
 \, \varrho_{\textup{irr}} \, \zeta_{J[L[\ell]],L(\ell)}
 \\& 
 \eqreff*{eq:stability}\le \,
 \varrho_{\textup{irr}} \, \eta_{L[\ell]} \,\zeta_{J[L[\ell]],L[\ell]} 
 + 
 \Cper \, \sum_{k=1}^{N-1}{\sf r}_{\ell-k}.
 \end{aligned}
\end{equation}

\textbf{Step~4 (goal-wise notational preliminaries).}
Note that the multigoal-estimator product $\Delta_\ell$ sums over the contributions of all dual problems. For the following step, it is important to change to a goal-wise viewpoint. For every goal $G_j$ with $j \in \{1, \dots, N\}$, there exists a unique $m = M[\ell, j] \in \{\ell + N-1, \dots, \ell + 2(N-1)\}$ such that $j = J[m]$, i.e., $m$ is the unique level such that the dual problem with index $j$ is active and hence $\zeta_{j, m} = \zetaa_{m}$. 
Thus, in case of irregular marking, define $\pi[j] \coloneqq J\big[L\big[M[\ell, j]\big]\big]$. The function $\pi$ maps the goal index $j$ to the active goal that is associated with the maximal estimator within the levels $\ell \le M[\ell, j]-(N-1)$ to $M[\ell, j] \le \ell + 2(N-1)$ (which will be helpful for irregular marking). Moreover, we define the pairwise disjoint sets that partition the indices of the goals $1 \le j \le N$ according to the algorithmic decision in the mark module, i.e.,
\begin{equation*}
   \Ireg[\ell] \coloneqq \{ j \in \{1, \dots, N\} | \text{ Algorithm~\ref{algorithm:Ngoal} uses regular marking in step } M[\ell, j] \}
\end{equation*}
and its complement $\Iirr[\ell] \coloneqq \{1, \dots, N\} \setminus \Ireg$ 
\begin{equation*}
   \Iirr[\ell]  
    = \{ j \in \{1, \dots, N\} | \text{ Algorithm~\ref{algorithm:Ngoal} uses irregular marking in step } M[\ell, j] \}.
\end{equation*}
Since $M[\ell, j] \in \{\ell + N-1, \dots, \ell + 2(N-1)\}$, this considers $N$ steps of the algorithm.  

\textbf{Step~5 (perturbed contraction $\boldsymbol{\ell+2N-1 \to \ell}$).}
In this step, we prove the perturbed contraction stated in~\eqref{eq:lemma:contraction:multigoal}.
Every dual problem with index $i = 1, \dots, N$ is active exactly once in the steps $\ell+N-1, \dots, \ell+2(N-1)$. 
In every step, either regular marking (discussed in Step~1) or irregular marking (discussed in Step~2) is applied.
Hence, for any $m \in \{\ell + N-1, \dots, \ell + 2(N-1)\}$, we have that 
\begin{subequations}\label{eq:contraction:multigoal}
 \begin{align}\label{eq:contraction:multigoal:case1}
  \eta_{m+1} \,\zeta_{J[m],m+1} 
  &\eqreff{eq:step1:convergence:multigoal}\leq
   \qest \, \eta_{m} \,\zeta_{J[m],m}
   + 
   \Cper \, {\sf r}_{m}
 \intertext{or}\label{eq:contraction:multigoal:case2}
  \eta_{m} \zeta_{J[m],m} 
  &\eqreff{eq:step2:convergence:multigoal}\leq
   \varrho_{\textup{irr}} \, \eta_{L[m]} \, \zeta_{J[L[m]],L[m]} 
   + 
   \Cper \, \sum_{k=1}^{N-1}{\sf r}_{m-k},
 \end{align}
\end{subequations}
where $L[m] \in \{m-1, \dots, m-N+1\}$ is the level from Step~1 in case of irregular marking.
Note that 
\begin{equation}\label{eq1:lemma:convergence:multigoal:step3}
 \sum_{k=1}^{N-1}{\sf r}_{m-k} 
 \leq 
 \sum_{k=\ell}^{\ell+2(N-1)} \!\! {\sf r}_k 
 \quad \text{for all } m \in \{\ell + N-1, \dots, \ell + 2(N-1)\}.
\end{equation}
We split up the multigoal-estimator product $\Delta_{\ell+2N-1}$ into contributions from regular and irregular marking with respect to the goals, i.e.,
\begin{equation}\label{eq:estimator:decomposition}
 \Delta_{\ell+2N-1} 
 = 
 \sum_{j=1}^N \eta_{\ell+2N-1} \,\zeta_{j,\ell+2N-1}
 = 
\mkern-3mu \sum_{j \in \Ireg[\ell]} \mkern-3mu \eta_{\ell+2N-1} \,\zeta_{j,\ell+2N-1} 
 + 
 \mkern-3mu \sum_{j \in \Iirr[\ell]} \mkern-3mu \eta_{\ell+2N-1} \,\zeta_{j,\ell+2N-1}.
\end{equation}
The regular terms in~\eqref{eq:estimator:decomposition} can be contracted with~\eqref{eq:contraction:multigoal:case1} for the active goals and stability~\eqref{eq:stability} for the inactive goals, yielding
\begin{equation}\label{eq:aux:case:regular}
\begin{aligned}
\sum_{j \in \Ireg[\ell]} \eta_{\ell+2N-1}\, \zeta_{j,\ell+2N-1} 
\quad &\stackrel{\mathclap{\eqref{eq:contraction:multigoal:case1}, \eqref{eq:stability}}}\le \quad
\qest \mkern-6mu\sum_{j \in \Ireg[\ell]} \mkern-6mu \eta_{\ell+N-1} \,\zeta_{j,\ell+N-1} 
+
|\Ireg[\ell]|\, \Cper \mkern-12mu \sum_{k=\ell+N-1}^{\ell+2(N-1)} \mkern-12mu {\sf r}_k 
\\
&\eqreff*{eq:stability}\le \,
\qest \mkern-6mu\sum_{j \in \Ireg[\ell]}\mkern-6mu \eta_{\ell} \,\zeta_{j,\ell}
+ 
|\Ireg[\ell]|\, \Cper \mkern-12mu \sum_{k=\ell}^{\ell+2(N-1)} \mkern-6mu {\sf r}_k.
   \end{aligned}
\end{equation}
The sum with irregular marking in~\eqref{eq:estimator:decomposition} is more involved and needs to be contracted down to level $\ell$ in one go. This is due to the fact that any irregular marking in the levels $\ell +N -1$ to $\ell + 2(N-1)$ compares the current level with the last $N-1$ levels. Thus, in our worst case analysis, one irregular term may be contracted back to level $\ell$. By exploiting~\eqref{eq:contraction:multigoal:case2} and~\eqref{eq:stability} repeatedly to cover the worst case, we deduce that
\begin{equation}\label{eq:aux:case:irregular}
\begin{aligned}
\sum_{j \in \Iirr[\ell]} \eta_{\ell+2N-1} \,\zeta_{j,\ell+2N-1} 
\le 
\varrho_{\textup{irr}} \mkern-6mu \sum_{j \in \Iirr[\ell]}\mkern-6mu \eta_{\ell}\, \zeta_{\pi[j],\ell} 
+ 
2 |\Iirr[\ell]|\, \Cper \mkern-12mu \sum_{k=\ell}^{\ell+2(N-1)} \mkern-12mu {\sf r}_k.
\end{aligned}
\end{equation}
The factor $2$ in front of $|\Iirr[\ell]|$ appears since we use stability~\eqref{eq:stability} for the inactive goals and the estimate~\eqref{eq:contraction:multigoal:case2} for the active goal already carries a sum (and we also use~\eqref{eq1:lemma:convergence:multigoal:step3}).

Overall, by combining~\eqref{eq:estimator:decomposition} with~\eqref{eq:aux:case:regular}--\eqref{eq:aux:case:irregular}, we have derived 
\begin{equation*}
 \Delta_{\ell+2N-1} 
 \le 
 \qest \mkern-6mu \sum_{j \in \Ireg[\ell]}\mkern-6mu  \eta_{\ell}\, \zeta_{j,\ell} 
 + 
 |\Ireg[\ell]|\, \Cper \mkern-12mu \sum_{k=\ell}^{\ell+2(N-1)} \mkern-12mu {\sf r}_k
 + \varrho_{\textup{irr}} \mkern-6mu \sum_{j \in \Iirr[\ell]} \mkern-6mu \eta_{\ell} \, \zeta_{\pi[j],\ell} 
 + 2 |\Iirr[\ell]|\, \Cper \mkern-12mu \sum_{k=\ell}^{\ell+2(N-1)}\mkern-12mu {\sf r}_k.
\end{equation*}
Note that each goal $G_j$ with index $j \in \{1, \dots, N\}$ appears at most once in the sum over $\Ireg[\ell]$ and at most $N-1$ times in the sum over $\Iirr[\ell]$ as $\pi[j] \neq j$.
Hence, we obtain that
\begin{equation*}
 \Delta_{\ell+2N-1} 
 \leq 
 \big[\qest + (N-1) \varrho_{\textup{irr}} \big] \Delta_{\ell}
  +
   2 N \,\Cper \!\!  \sum_{k=\ell}^{\ell+2(N-1)} \!\! {\sf r}_k.
\end{equation*}

\textbf{Step~6 (summability of the remainder term).}
Set $\Rho_{k} \coloneqq 2N \,\Cper \sum_{k'=k}^{k+2(N-1)} {\sf r}_{k'}$. 
Let $n \in \N_0$ be arbitrary. 
We observe that every term ${\sf r}_k$  appears at most $2N-1$ times in the sum $\sum_{k=\ell}^{\ell+n} \Rho_k^2$.
With the discrete Cauchy--Schwarz inequality, we thus obtain that
\begin{equation*}
 \sum_{k=\ell}^{\ell+n} \Rho_k^2 
 \leq 
 (2N)^3\, \Cper^2 \sum_{k=\ell}^{\ell+n}\sum_{k'=k}^{k+2(N-1)} {\sf r}_{k'}^2
  \leq 
  (2N)^4 \,\Cper^2 \mkern-12mu \sum_{k=\ell}^{(\ell+n)+2(N-1)} \mkern-12mu  {\sf r}_k^2.
\end{equation*}
The discrete Cauchy--Schwarz inequality also yields
\begin{equation*}
 {\sf r}_k^2
  \leq 
  2N \Big( \eta_{k}^2\, \sum_{i=1}^N \enorm{z_{i,k+1}^{\star} - z_{i,k}^{\star}}^2 
  + 
  \enorm{u_{k+1}^{\star} - u_k^{\star}}^2\, \sum_{i=1}^N \zeta_{i,k}^2 \Big).
\end{equation*}
Using quasi-monotonicity~\eqref{axiom:qm} and $K \coloneqq \ell+n+2(N-1)$, we derive that 
\begin{equation*}
 \sum_{k=\ell}^{\ell+n} \Rho_k^2 
 \leq 
 (2N)^5 \, \Cper^2\, \Cmon^2 \Big[ \eta_{\ell}^2 \,\sum_{i=1}^N \sum_{k=\ell}^{K}\enorm{z_{i,k+1}^{\star} - z_{i,k}^{\star}}^2 
 + 
 \Big(\sum_{i=1}^{N} \zeta_{i,\ell}^2\Big) \,\sum_{k=\ell}^{K} \enorm{u_{k+1}^{\star}- u_k^{\star}}^2\Big].
\end{equation*}
Recall that $\enorm{v}^2 = a(v,v)$ and that the variational formulations~\eqref{eq:weak_formulation} and~\eqref{eq:discrete} together with $\XX_{\ell} \subseteq \XX_{\ell+1}$ imply the Galerkin orthogonality
\begin{equation*}
   a(u^{\star} - u_{\ell+1}^{\star}, u_{\ell+1}^\star-u_\ell^\star) 
   = 0 
   \quad \text{ for all } \ell \in \N_0.
\end{equation*}
The resulting Pythagorean identity
\begin{equation*}
 \enorm{u_{k+1}^{\star} - u_k^{\star}}^2 
 = 
 \enorm{u^{\star} - u_k^{\star}}^2 
 - 
 \enorm{u^{\star} - u_{k+1}^{\star}}^2
\end{equation*}
thus yields the telescoping sum
\begin{equation*}
 \sum_{k=\ell}^{K} \enorm{u_{k+1}^{\star} - u_k^{\star}}^2 
 = 
 \enorm{u^{\star} - u_{\ell}^{\star}}^2 
 - 
 \enorm{u^{\star} - u_{K+1}^{\star}}^2 
 \leq 
  \enorm{u^{\star} - u_{\ell}^{\star}} 
  \leq 
  \Crel^2\, \eta_{\ell}^2.   
\end{equation*}
The same argument applies to the dual problems.
Overall, we thus derive that
\begin{equation*}
 \sum_{k=\ell}^{\ell+n} \Rho_k^2 
 \leq
  2^6 \, N^5 \,\Cper^2\, \Cmon^2 \,\Crel^2  \, \eta_{\ell}^2 \sum_{i=1}^N \zeta_{i,\ell}^2
  \eqqcolon 
  \Csum \, \eta_{\ell}^2 \,\sum_{i=1}^N \zeta_{i,\ell}^2 
   \leq
    \Csum \, \Delta_\ell^2.
\end{equation*}
Since $\Csum$ is independent of $n$, this concludes the proof of~\eqref{eq:lemma:contraction:multigoal}.
\end{proof}

\begin{remark}
To reduce the $N$-dependence in the estimate, one might consider the $\ell_2$-sum of the goal errors $\sum_{j=1}^{N}|G_j(u^{\star})- G_j(u_{\ell}^{\star})|^2$ with $\widetilde{\Delta}_{\ell} \coloneqq \eta^2_{\ell}\, \sum_{i=1}^N \zeta^2_{i,\ell}$. 
In this case, stability~\eqref{eq:stability} and reduction~\eqref{eq:reduction} are employed using the Young inequality such that
\begin{equation*}
   \mu_h(v_h)^2 \leq 
   (1+\delta) \, \mu_H(v_h)^2 + (1+\delta^{-1}) \, \Cstab^2 \, \enorm{v_h - v_H}^2
\end{equation*}
as well as
\begin{equation*}
   \mu_h(v_h)^2 \leq 
  (1+\delta) \, \qest^2 \, \mu_H(v_H)^2 + (1+\delta^{-1}) \, \Cstab^2 \, \enorm{v_h - v_H}^2.
\end{equation*}
Define $\widetilde{C}_{\textnormal{stab}} \coloneqq (1+\delta^{-1}) \, \Cstab^2$ and $\widetilde{C}_{\textnormal{per}} \coloneqq 2 \widetilde{C}_{\textnormal{stab}} + (1+\Cmon)^2 \, \widetilde{C}_{\textnormal{stab}}^2 \, \Crel^2$. 
Choosing $0<\delta$ sufficiently small such that $(1+ \delta) \, \qest^2 < 1$ and $(1+\delta)^N < 2$, the perturbed contraction reads
\begin{equation*}
 \widetilde{\Delta}_{\ell+2N-1} 
 \leq
 \widetilde{q}_{\textnormal{ctr}} \, \widetilde{\Delta}_{\ell} 
 + \widetilde{\Rho}_{\ell}
 \quad \text{with} \quad
 \sum_{k=\ell}^\infty \widetilde{\Rho}_k 
 \leq \widetilde{C}_{\textnormal{sum}} \, \widetilde{\Delta}_{\ell},
\end{equation*}
where $\widetilde{q}_{\textnormal{ctr}} \coloneqq (1+\delta) \, \qest^2 + (N-1) (1+\delta)^{N-1} \, \varrho_{\textup{irr}}^2$ and $\widetilde{C}_{\textnormal{sum}} \coloneqq 2^{3} \, N^2 \, \widetilde{C}_{\textnormal{per}}^2 \, \Cmon^2 \, \Crel^2$ and a suitable adaptation of the remainder $\widetilde{\Rho}_\ell$.
\end{remark}

\begin{proof}[\bfseries Proof of Theorem~\ref{theorem:multigoal:linearConvergence}]
Together with~\eqref{eq:lemma:contraction:multigoal}, the Young inequality yields
\begin{equation}\label{eq:theorem:linconv:contraction}
 \Delta_{\ell+2N-1}^2
  \le 
  (1+\delta) \,\qctr \, \Delta_\ell^2 
  + 
  (1 + \delta^{-1}) \, \Rho_\ell^2
 \quad \text{ for all } \ell \in \N_0 \text{ and all } 0 <\delta.
\end{equation}
We choose $0 <\delta$ sufficiently small such that $q \coloneqq (1+\delta)\, \qctr < 1$.
We aim to verify the summability criterion from~\cite[Lemma~2]{bfmps2025} for $a_\ell \coloneqq \Delta_\ell$. To this end, let $\ell, n \in \N_0$.
Together with quasi-monotonicity~\eqref{axiom:qm}, the last formula~\eqref{eq:theorem:linconv:contraction} leads to
\begin{equation*}
   \begin{aligned}
 \sum_{k = \ell}^{\ell + n} \Delta_k^2 
 \mkern5mu&  \eqreff*{axiom:qm}\le \mkern5mu
 [1 + 2(N-1)\,\Cmon^{4}] \, \Delta_\ell^2 
 + 
 \mkern-12mu \sum_{k = \ell+2N-1}^{\ell + n}\mkern-12mu  \Delta_k^2
 = 
  [1 + 2(N-1)\,\Cmon^{4}] \, \Delta_\ell^2 
 + 
\mkern-18mu \sum_{k = \ell}^{\ell + n-(2N-1)} \mkern-18mu \Delta_{k+2N-1}^2 
\\&  
\eqreff*{eq:theorem:linconv:contraction}\leq
 [1 + 2(N-1)\,\Cmon^{4}] \, \Delta_\ell^2 
  +
   \sum_{k=\ell}^{\ell + n - (2N-1)} \Big( q \, \Delta_k^2 + (1+\delta^{-1}) \, \Rho_k^2 \Big).
   \end{aligned}
\end{equation*}
Rearranging this estimate, we see that
\begin{equation*}
   \begin{aligned}
 \sum_{j = \ell}^{\ell + n} \Delta_j^2
 & \; \le\; \frac{1}{1-q} \, \Big( [1 + 2(N-1)\,\Cmon^{4}] \, \Delta_\ell^2 + (1+\delta^{-1}) \, \sum_{k = \ell}^{\ell + n - (2N-1)} \Rho_k^2 \Big)
 \\& \;
 \eqreff*{eq:lemma:contraction:multigoal}\le \;
 \frac{(1 + 2(N-1)\,\Cmon^4) + (1+\delta^{-1})\,\Csum}{1-q} \, \Delta_\ell^2 
 \eqqcolon
  \Ctail\, \Delta_{\ell}^2.
\end{aligned}
\end{equation*}
Therefore,~\cite[Lemma~2]{bfmps2025} concludes the proof of R-linear convergence~\eqref{eq:multigoal:linearConvergence} with $\Clin^2 \coloneqq 1 + \Ctail$ and $\qlin^2 \coloneqq (1 + \Ctail)/(2 + \Ctail)$ (as the proof in~\cite{bfmps2025} reveals).
\end{proof}

%% file: 04_ratesN.tex
\section{Optimal convergence rates}\label{section:ratesN}
Optimal convergences rates are captured by means of nonlinear approximation classes as introduced in~\cite{bdd2004} in the context of adaptive FEM. The main result of this paper (Theorem~\ref{theorem:multigoal:optimalRates} below) connects the decay of the multigoal-estimator product $\Delta_\ell$ defined in~\eqref{eq:quasi-error:multigoal} to the best possible rates that can be achieved by a theoretical sequence of optimal meshes. To this end, for $n \in \N_0$, we write $\TT_{H} \in \T_n$ if $\TT_{H}$ is a refinement of $\TT_0$ with at most $n$ more elements, i.e., $\TT_{H} \in \T$ and $\# \TT_{H} - \# \TT_0 \le n$. We say that $w^\star \in \{u^\star, z_1^\star, \ldots, z_N^\star\}$ is in the nonlinear approximation class of rate $0<r$, if
\begin{equation}\label{eq:approximationClass}
  \norm{w^\star}_{\mathbb{A}_r}
  \coloneqq
  \sup_{n \in \N_0} \big( (n+1)^r \, \min_{\TT_{\textup{opt}} \in \T_n} \mu_{\textup{opt}}(w^\star) \big)
  < \infty,
\end{equation}
where $\mu_{\textup{opt}}(w^\star)$ is the corresponding error estimator from~\eqref{eq:estimators} associated with the optimal triangulation $\TT_{\textup{opt}} \in \T_n$ and in accordance with $w^\star$. While our definition of $\norm{w^\star}_{\mathbb{A}_r}$ follows~\cite{axioms} and relies on the estimator only, we note that $\norm{w^\star}_{\mathbb{A}_r}$ can indeed be characterized in terms of error plus data oscillations as in~\cite{ckns2008}. We refer, e.g., to~\cite{axioms} for details.

The following main theorem gives meaningful results for the case where the primal and dual problems can be approximated with certain rates $0< s,t$ in the sense that 
\[
\norm{u^\star}_{\mathbb{A}_s} +  \sum_{i=1}^N \norm{z_i^\star}_{\mathbb{A}_t}  < \infty,
\]
i.e., there are (potentially different) mesh sequences such that the primal estimator $\eta$ decays (at least) with rate $s$ and all dual estimators $\zeta_{j}$ decay (at least) with rate $t$. According to the following theorem, the proposed Algorithm~\ref{algorithm:Ngoal} then guarantees that the multigoal-estimator product $\Delta_\ell$ from~\eqref{eq:quasi-error:multigoal} decays with rate $s+t$.

\begin{theorem}[optimal convergence rates of NGO-AFEM]\label{theorem:multigoal:optimalRates}
  Suppose stability~\eqref{axiom:stability}, reduction~\eqref{axiom:reduction}, discrete reliability~\eqref{axiom:discrete_reliability}, and quasi-monotonicity~\eqref{axiom:qm}.
  Define $\theta_\star \coloneqq 1/(1 + \Cstab^2\Cdrel^2)$ and
  \begin{equation}\label{eq:lemma:marking:multigoal:constant}
    C_{0} 
    \coloneqq
     \max_{k=0, \dots, N-2} \# \MM_{k} \, \Delta_{k}^{1/(s+t)}.
  \end{equation}
  Let $0 < \theta < \theta_\star$ and $0 < \varrho_{\textup{irr}} < (1- \qest)/(N-1)$, where $0< \qest = [1-(1-\qred^2)\,\theta]^{1/2} <1$ as in Lemma~\ref{lemma:stability+reduction}.
  Then, Algorithm~\ref{algorithm:Ngoal} guarantees, for all rates $0 <s, t$, that
  \begin{equation}\label{eq:multigoal:thm:optimal}
    \sup_{\ell \in \N_0}(\#\TT_\ell - \#\TT_0 + 1)^{s+t} \Delta_{\ell}
    \le C_{\rm opt} \, \max \Big\{C_0,\max_{j=1,\dots,N} \big[ \norm{u^\star}_{\mathbb{A}_s} \norm{z_j^\star}_{\mathbb{A}_t} \big]^{1/(s+t)} \Big\},
  \end{equation}
  where $C_{\rm opt} > 0$ depends only on the constants in~\eqref{axiom:stability},~\eqref{axiom:reduction},~\eqref{axiom:discrete_reliability},~\eqref{axiom:qm}, and on $\Cmark$, $\theta$, $\varrho_{\textup{irr}}$, $s$, $t$, $N$, the initial mesh $\TT_0$, and the mesh-closure constant.
\end{theorem}
This statement can be interpreted as follows. If rate $0 < s$ is theoretically possible for the primal problem, i.e., $\norm{u^\star}_{\mathbb{A}_s} < \infty$, as well as rate $0<t$ is possible for all dual problems, i.e., $\sum_{i=1}^N \norm{z_i^\star}_{\mathbb{A}_t} < \infty$, then Algorithm~\ref{algorithm:Ngoal} drives down the multigoal-estimator product $\Delta_\ell$ from~\eqref{eq:quasi-error:multigoal} with rate $0 < s+t$. More formally, there holds $\Delta_\ell \lesssim (\dim \XX_\ell)^{-(s+t)}$ with $\dim \XX_\ell$ being the number of degrees of freedom of the FE discretization.   

The mathematical core of Theorem~\ref{theorem:multigoal:optimalRates} is the following lemma, which controls the number of marked elements in each step of the adaptive algorithm.
\begin{lemma}[control over the number of marked elements]\label{lemma:marking}
Suppose stability~\eqref{axiom:stability} and discrete reliability~\eqref{axiom:discrete_reliability}, and let $0 < \theta < \theta_\star = 1/(1 + \Cstab^2\Cdrel^2)$ as in Theorem~\ref{theorem:multigoal:optimalRates}.
Then, there exists a constant $C_1 > 0$ such that, for all $0 < s, t$ and all $\ell \in \N_0$, the sets of marked elements $\MM_\ell$ of Algorithm~\ref{algorithm:Ngoal} satisfy with $C_0$ from~\eqref{eq:lemma:marking:multigoal:constant} that
\begin{equation}\label{eq:lemma:marking}
 \#\MM_\ell 
 \le 
C_1 \, \max \Big\{ C_{0}, \big[ \norm{u^\star}_{\mathbb{A}_s}\, \norm{z_j^\star}_{\mathbb{A}_t} \big]^{1/(s+t)} \Big\}  \,
 \Delta_{\ell}^{-1/(s+t)}, \quad \text{ where } j= {\rm mod}(\ell,N) + 1.
\end{equation}
The constant $C_1$ depends only on $\Cmark$, $\Cmon$, $q_\star$, $\varrho_{\textup{irr}}$, $s$, $t$, and $N$.
\end{lemma}
The proof of Lemma~\ref{lemma:marking} on the control  the
number of marked elements in Algorithm~\ref{algorithm:Ngoal} requires the following two auxiliary results from the literature. 

\begin{lemma}[{optimality of D\"orfler marking~\cite[Proposition~4.12]{axioms}}]\label{lemma:doerfler}
  Suppose stability~\eqref{axiom:stability} and discrete reliability~\eqref{axiom:discrete_reliability}, and
  recall $0 < \theta_\star < 1$ from Theorem~\ref{theorem:multigoal:optimalRates}.
  Then, for all $0 < \theta < \theta_\star$, there exists $0 < q_\star < 1$ such that for all $\TT_H \in \T$
  and all $\TT_h \in \T(\TT_H)$, there holds the following implication for any of the estimators $\mu_H \in
  \{\eta_H, \zeta_{1,H}, \ldots, \zeta_{N,H}\}$:
  \begin{equation}\label{eq:lemma:doerfler}
    \mu_h \le q_\star \mu_H
    \quad \Longrightarrow \quad
    \theta \, \mu_H^2 \le \mu_H(\TT_H \backslash \TT_h)^2.
    \qquad \qed
  \end{equation}
\end{lemma}
The definition of the approximation classes~\eqref{eq:approximationClass} and the validity of the overlay estimate for newest-vertex bisection from~\cite{s2007, ckns2008} enable the following lemma.

\begin{lemma}[{comparison lemma~\cite[Lemma~14]{ffghp2016}}]\label{lemma:comparison}
  Given any $0 < q < 1$, each mesh $\TT_H \in \T$ admits a
  refinement $\TT_h \in \T(\TT_H)$ such that for all $0 < s, t$ and all $j = 1, \dots, N$, it holds that
  \begin{equation}
    \begin{aligned}
     \eta_h\, \zeta_{j,h}
    & \le q \, \eta_H \,\zeta_{j,H} \quad \text { and }
    \\
    \#\TT_h - \#\TT_H
    & \le
    2 \,\big[\Ccomp \, q^{-1/4} \,\norm{u^\star}_{\mathbb{A}_s}\, \norm{z_j^\star}_{\mathbb{A}_t}\big]^{1/(s+t)} \,
    \big[ \eta_H \,\zeta_{j,H} \big]^{-1/(s+t)}. \hfill \qed
    \end{aligned}
  \end{equation}
\end{lemma}

\begin{proof}[\textbf{Proof of Lemma~\ref{lemma:marking}}]
The proof is split into four steps.
Step~1 provides control on the set $\MM_\ell^{uz}$.
Steps~2--3 concern \emph{regular marking} in Algorithm~\ref{algorithm:Ngoal}\eqref{algorithm:mark}, where $\MM_\ell = \MM_\ell^{uz}$ satisfies the D\"orfler marking criterion for $\eta_\ell$ or $\zetaa_\ell$.
Step~4 covers the case of \emph{irregular marking} in Algorithm~\ref{algorithm:Ngoal}\eqref{algorithm:mark}, where D\"orfler marking might fail since $\MM_\ell \subsetneqq \MM_\ell^{uz}$.

\textbf{Step~1 (control of $\boldsymbol{\#\MM_\ell^{uz}}$ for all $\boldsymbol{\ell \in \N_0}$).}
Recall that $\zetaa_\ell = \zeta_{j,\ell}$ with $j = {\rm mod}(\ell,N)+1$.
For $q = q_\star^2$ and $\TT_H = \TT_\ell$, the mesh $\TT_h \in \T(\TT_\ell)$ from Lemma~\ref{lemma:comparison} guarantees
\begin{subequations}\label{eq1:proof:rates:multigoal}
\begin{align}\label{eq1a:proof:rates:multigoal}
 \eta_h \zeta_{j,h}
 & \le q_\star^2 \, \eta_\ell \zeta_{j,\ell} \quad \text { and }
 \\ \label{eq1b:proof:rates:multigoal}
 \#\TT_h - \#\TT_\ell
 & \le
 2\, \big[\Ccomp \,q_\star^{-1/2}\, \norm{u^\star}_{\mathbb{A}_s}\, \norm{z_j^\star}_{\mathbb{A}_t}\big]^{1/(s+t)} \,
 \big[ \eta_\ell \zeta_{j,\ell} \big]^{-1/(s+t)}.
 \end{align}
\end{subequations}
The first inequality~\eqref{eq1a:proof:rates:multigoal} yields that
\begin{equation*}
 \eta_h \le q_\star \, \eta_\ell
 \quad \text{or} \quad
 \zeta_{j,h} \le q_\star \, \zeta_{j,\ell}.
\end{equation*}
Lemma~\ref{lemma:doerfler} and the quasi-minimal choice of $\MM_\ell^u$ and $\MM_\ell^z$ thus guarantee
\begin{equation*}
 \#\MM_\ell^u \le \Cmark \, \# (\TT_\ell \backslash \TT_h)
 \quad \text{or} \quad
 \#\MM_\ell^z \le \Cmark \, \# (\TT_\ell \backslash \TT_h).
\end{equation*}
In any case, this proves that
\begin{equation}\label{eq:step1:lemma:marking:multigoal}
 \begin{split}
  \#\MM_\ell^{uz}
 & \; \le \;  \Cmark \, \min\{ \#\MM_\ell^u , \#\MM_\ell^z \}
 \le \Cmark^2 \, \# (\TT_\ell \backslash \TT_h)
 \le \Cmark^2 \, ( \#\TT_h - \#\TT_\ell )
 \\& 
 \; \eqreff*{eq1b:proof:rates:multigoal}\le \;  
 2 \, \Cmark^2 \, \big[ \Ccomp \,q_\star^{-1/2} \,\norm{u^\star}_{\mathbb{A}_s} \,\norm{z_j^\star}_{\mathbb{A}_t} \big]^{1/(s+t)}\,
 \big[ \eta_\ell\, \zeta_{j,\ell} \big]^{-1/(s+t)}.
 \end{split}
\end{equation}

\textbf{Step~2 (control of $\boldsymbol{\#\MM_\ell}$ for $\boldsymbol{\ell \leq N-2}$).}
From the definition~\eqref{eq:lemma:marking:multigoal:constant} of $C_{0}$, it immediately follows that~\eqref{eq:lemma:marking} is satisfied with $C_1 = 1$ for $\ell \leq N-2$.

\textbf{Step~3 (control of $\boldsymbol{\#\MM_\ell}$ for regular marking and $\boldsymbol{N-1 \le \ell}$).}
We have that $\varrho_{\textup{irr}} \, \max_{i=1,\dots, N-1} \zetaa_{\ell-i} \le \zetaa_{\ell}$.
Quasi-monotonicity~\eqref{axiom:qm} shows that
\begin{equation*}
 \zeta_{J[\ell-i],\ell} 
 \leq 
 \Cmon \, \zeta_{J[\ell-i],\ell-i} 
 = 
 \Cmon \,\zetaa_{\ell-i}.
\end{equation*}
From this, it follows that
\begin{equation*}
 \sum_{i=0}^{N-1} \zeta_{i,\ell}=
  \sum_{i=0}^{N-1} \zeta_{J[\ell-i],\ell}
 \le 
 \zetaa_\ell + \Cmon \sum_{i=1}^{N-1} \zetaa_{\ell-i}
 \le 
 \big(1 + (N-1) \Cmon \,\varrho_{\textup{irr}}^{-1}\big) \, \zetaa_\ell.
\end{equation*}
Together with
\begin{equation}\label{eq:step3:lemma:marking:multigoal}
  C_1' \coloneqq 2 \, \Cmark^2 \,
  \big[\Ccomp\, q_\star^{-1/2}\,
   \big(1+ (N-1)\,\Cmon\, \varrho_{\textup{irr}}^{-1}\big)\big]^{1/(s+t)}
\end{equation} 
and~\eqref{eq:step1:lemma:marking:multigoal}, this yields
\begin{equation*}
  \begin{aligned}
 \#\MM_\ell^{uz} 
 & \; \eqreff*{eq:step1:lemma:marking:multigoal} \le \;
 2 \, \Cmark^2 \, \big[\Ccomp \, q_\star^{-1/2} \, \norm{u^\star}_{\mathbb{A}_s}\, \norm{z_j^\star}_{\mathbb{A}_t}\big]^{1/(s+t)}\,
 \big[ \eta_\ell \, \zetaa_{\ell} \big]^{-1/(s+t)}
 \\&
 \; \le \;  C'_1 \, \big[\norm{u^\star}_{\mathbb{A}_s} \,\norm{z_j^\star}_{\mathbb{A}_t}\big]^{1/(s+t)}\,
 \Delta_{\ell}^{-1/(s+t)}.
  \end{aligned}
\end{equation*}
This proves~\eqref{eq:lemma:marking} with $C_1 = \max \{1, C_1'\}$ for any $\ell \in \N$ with regular marking.
Overall,~\eqref{eq:lemma:marking} is thus proved for all $\ell \in \N_0$ with regular marking.

\textbf{Step~4 (control of $\boldsymbol{\#\MM_\ell}$ for irregular marking).}
Let $N-1 \le \ell$ and we have that there exists an index $i \in \{1,\dots, N-1\}$ such that $\zetaa_\ell < \varrho_{\textup{irr}} \, \zetaa_{\ell-i}$.
Let $0 \le \ell' < \ell$ be the largest index such that $\MM_{\ell'}$ is obtained by regular marking.
Note that by choice of $\ell'$ all mesh levels $k \in \{ \ell' +1, \dots, \ell\}$ employ irregular marking.
Hence, Algorithm~\ref{algorithm:Ngoal}\eqref{algorithm:mark} ensures that
\begin{equation*}
 \# \MM_\ell \le \dots \le \#\MM_{\ell'}
 \eqreff{eq:lemma:marking}\le 
 \max \Big\{C_{0}, C_1' \,\big[ \norm{u^\star}_{\mathbb{A}_s} \norm{z_j^\star}_{\mathbb{A}_t} \big]^{1/(s+t)}\Big\} \,
 \Delta_{\ell'}^{-1/(s+t)},
\end{equation*}
where $0 <C_1'$ is the constant defined in~\eqref{eq:step3:lemma:marking:multigoal}.
With quasi-monotonicity~\eqref{axiom:qm} (or $R$-linear convergence~\eqref{eq:multigoal:linearConvergence}), we verify that 
\begin{equation*}
\Delta_{\ell} \le \Cmon^2 \, \Delta_{\ell'}.
\end{equation*}
Combining the last two formulas, we prove~\eqref{eq:lemma:marking} with $C_1 = \Cmon^{2/(s+t)} \max \{1, C_1'\}$ for any $\ell \in \N$ with irregular marking. This concludes the proof.
\end{proof}

\begin{proof}[\bfseries Proof of Theorem~\ref{theorem:multigoal:optimalRates}]
Let $\ell \in \N$.
Recall the mesh-closure estimate from~\cite{bdd2004},
\begin{equation}\label{eq:mesh-closure}
\#\TT_\ell - \#\TT_0 \le C_\textup{mesh}\, \sum_{k=0}^{\ell-1} \#\MM_k
\end{equation} 
 with a constant $C_\textup{mesh}$ that depends only on $d$ and $\TT_0$. With this, the control on the number of marked elements~\eqref{eq:lemma:marking}, and R-linear convergence~\eqref{eq:multigoal:linearConvergence}, it follows that
\begin{equation*}
  \begin{aligned}
& \frac{1}{2}\, ( \#\TT_\ell - \#\TT_0 + 1)
 \le 
 \#\TT_\ell - \#\TT_0
 \eqreff{eq:mesh-closure}\le 
 C_\textup{mesh} \,
 \sum_{k=0}^{\ell-1} \#\MM_k
 \\&
 \eqreff{eq:lemma:marking}\lesssim
 \max \Big\{C_{0},\max_{j=1, \dots, N} \big[ \norm{u^\star}_{\mathbb{A}_s} \norm{z_j^\star}_{\mathbb{A}_t} \big]^{1/(s+t)} \Big\}  
 \Bigg[ \sum_{k=0}^{\ell-1} \Delta_k^{-1/(s+t)} \Bigg]
 \\&
 \eqreff{eq:multigoal:linearConvergence}\lesssim
 \max \Big\{C_{0},\max_{j=1, \dots, N} \big[ \norm{u^\star}_{\mathbb{A}_s} \,\norm{z_j^\star}_{\mathbb{A}_t} \big]^{1/(s+t)} \Big\}
    \Delta_{\ell}^{-1/(s+t)}.
  \end{aligned}
  \end{equation*}
  Rearranging this estimate, we see that, for all $\ell \in \N$,
  \begin{equation*}
    (\#\TT_\ell - \#\TT_0 + 1)^{s+t} \Delta_{\ell}
    \lesssim \max \Big\{C_{0},\max_{j=1, \dots, N} \big[ \norm{u^\star}_{\mathbb{A}_s} \norm{z_j^\star}_{\mathbb{A}_t} \big]^{1/(s+t)} \Big\}. 
  \end{equation*}
 Noting that this estimate is trivial for $\ell = 0$ by definition of $\norm{u^\star}_{\mathbb{A}_s}$ and $\norm{z_j^\star}_{\mathbb{A}_t}$,
 we conclude the proof.
\end{proof}

  We remark that an initial solve (and comparison) of all dual problems on the inexpensive initial level $\ell=0$ allows to improve the constant $C_0$ from~\eqref{eq:lemma:marking:multigoal:constant} in Theorem~\ref{theorem:multigoal:optimalRates} as follows.
\begin{corollary}\label{cor:multigoal:optimalRates}
  Let us adapt Algorithm~\ref{algorithm:Ngoal} such that on level $\ell=0$ all dual problems are solved and potentially renumbered such that $\zeta_{N,0} \le \zeta_{N-1,0} \le \ldots \le \zeta_{1,0}$.
  Moreover, we set $\zetaa_{\ell-i} \coloneqq \zeta_{i+1,0}$ for $\ell < i$ in the initialization of Algorithm~\ref{algorithm:Ngoal}.
  Under the assumptions of Theorem~\ref{theorem:multigoal:optimalRates} and for all $0 <s, t$, the optimality result~\eqref{eq:multigoal:thm:optimal} holds in the improved form
  \begin{equation}\label{eq:multigoal:cor:optimal}
    \sup_{\ell \in \N_0}(\#\TT_\ell - \#\TT_0 + 1)^{s+t} \Delta_{\ell}
    \le C_{\rm opt} \, \max_{j=1,\dots,N} \big[ \norm{u^\star}_{\mathbb{A}_s} \norm{z_j^\star}_{\mathbb{A}_t} \big]^{1/(s+t)},
  \end{equation}
  i.e., the constant $C_0$ from~\eqref{eq:lemma:marking:multigoal:constant} can be omitted since there holds
  \begin{equation*}
    C_0 \leq \widetilde{C}_1 \big[ \norm{u^\star}_{\mathbb{A}_s} \norm{z_j^\star}_{\mathbb{A}_t} \big]^{1/(s+t)}.
  \end{equation*}
  The constant $\widetilde{C}_1$ depends only on $\Cmark$, $\Cmon$, $q_\star$, $\varrho_{\textup{irr}}$, and $N$.
\end{corollary}

\begin{proof}
  We adapt the proof of Lemma~\ref{lemma:marking} and focus on $\ell = 0$.
  Since regular marking is performed on level $\ell=0$, the estimate~\eqref{eq:lemma:marking} from the proof of Lemma~\ref{lemma:marking} applies and yields
  \begin{equation*}
    \begin{aligned}
    \# \MM_0 = \#\MM_0^{uz} 
    &\le 
    2 \, \Cmark^2 \, \big[\Ccomp \, q_\star^{-1/2}\, \norm{u^\star}_{\mathbb{A}_s} \, \norm{z_1^\star}_{\mathbb{A}_t}\big]^{1/(s+t)}\,
    \big[ \eta_0 \, \zeta_{1,0}\,  \big]^{-1/(s+t)}.
    \end{aligned}
  \end{equation*}
  Moreover, by the potential renumbering of the dual problems, we have
  \begin{equation*}
   \Delta_0 \leq N \,\eta_0\, \zeta_{1,0}.
  \end{equation*}
  The control of $\# \MM_{\ell}$ for any $0 \le \ell$ follows from the proof of Lemma~\ref{lemma:marking}, where we can omit the constant $C_0$ by virtue of the last estimate.
  This concludes the proof.
\end{proof}

%% file: 05_numerics.tex
\section{Numerical experiments}\label{section:numerics}
In this section, we present several numerical experiments that underline the theoretical findings of this work. In particular, we verify the optimal convergence rates of the proposed multi-goal adaptive algorithm (Algorithm~\ref{algorithm:Ngoal}) for several experimental setups with unknown exact solutions and decoupled goal functionals in the dual problems. 
All experiments are performed with the open-source \textsc{Matlab} library MooAFEM~\cite{ip2022}.

\subsection{Experiment with \texorpdfstring{$\boldsymbol{N=3}$}{N=3} goals}\label{experiment:smooth}

We consider a variant of~\cite[Example~7.3]{ms2009} on the domain $\Omega = (0,1)^2 \subset \R^2$ with homogeneous Dirichlet boundary conditions. The diffusion coefficient $\boldsymbol{A} = 1$, and the right-hand sides $f=0$ and $\boldsymbol{f} = (-1, 0)^\top \chi_{\Omega_1}$. The goals are defined as $g_1= g_2 = g_3 = 0$, and $\boldsymbol{g_1} = (1,0)^\top \chi_{\Omega_2}$,  $\boldsymbol{g_2} = (1,0)^\top \chi_{\Omega_3}$, and $\boldsymbol{g_3} = (0,1.5)^\top \chi_{\Omega_4}$, where $\chi_{\Omega_i}$ denotes the characteristic function on the subdomain $\Omega_i$ defined as in Figure~\ref{figure:smoothSolutionMeshes}. The adaptivity parameters are set to $\theta = 0.5$, $\Cmark = 2$, and $\varrho_{\textup{irr}} = 0.25$. 
If not stated otherwise, we use Algorithm~\ref{algorithm:Ngoal}, where we select exactly $\# \MM_{\ell-1}$ elements from the set $\MM_{\ell}^{uz}$ in the irregular marking step (\ref{algorithm:mark}).
Figure~\ref{figure:smoothSolutionMeshes} also displays snapshot meshes at a certain point in the algorithm for both $p=1$ (left) and $p=3$ (right).

Figure~\ref{figure:smoothSolutionConvergence} shows convergence rates for the multigoal-error product $\Delta_\ell$ from~\eqref{eq:quasi-error:multigoal} for $p \in \{1,2,3\}$ (left) and the individual error estimators $\eta_{\ell}$ and $\zeta_{i, \ell}$ with $i=1,2,3$ for $p=2$ (right). Optimal convergence rates $\mathcal{O}(\texttt{nDof}^{-p})$ for the multigoal-error estimator and $\mathcal{O}(\texttt{nDof}^{-p/2})$ for the individual estimators are attained. 
Note that the values of the error estimators $\zeta_{i,\ell}$ do not change when they are inactive, even though the mesh continues to be refined, since the inactive estimators are not recomputed. Thus, the zoom-in in Figure~\ref{figure:smoothSolutionConvergence} (right) exhibits a staircase-like behavior.

Figure~\ref{figure:smoothSolutionConvergenceAFEM} compares the multigoal-error estimator of NGO-AFEM with other adaptive schemes. The left panel investigates standard AFEM, i.e., skipping step~(\ref{algorithm:solve:dual}) and step~(\ref{algorithm:markall}) in Algorithm~\ref{algorithm:Ngoal} such that $\MM_{\ell} = \MM_{\ell}^{u}$ for all $\ell \in \N_0$. The right panel depicts the case where NGO-AFEM is run with $N=2$, i.e., goal $G_3$ is never active, and step~(\ref{algorithm:solve:dual}) in Algorithm~\ref{algorithm:Ngoal} solves only for $j \in \{1,2\}$. In both cases, only NGO-AFEM recovers optimal convergence rates, while the two variants perform suboptimally. In case of standard AFEM, the convergence rate of all dual estimators $\zeta_{i,\ell}$ for $i=1,2,3$ is reduced, since the mesh is only refined with respect to the primal problem. For the NGO-AFEM with $N=2$, the dual estimator $\zeta_{3,\ell}$ is not resolved properly and is thus suboptimally convergent. This underlines the importance of resolving the singularities of all dual problems for optimal convergence rates of the multigoal-error estimator $\Delta_\ell$.

\begin{figure}
\includegraphics[height=0.3\textwidth]{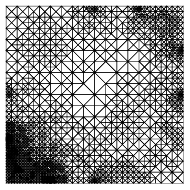} \hfil
\includegraphics[height=0.3\textwidth]{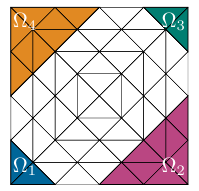} \hfil
\includegraphics[height=0.3\textwidth]{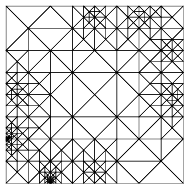}
\caption{Meshes for the experiment from section~\ref{experiment:smooth}: Initial mesh (middle) and results for $p=1$ (left) with $\texttt{nDof} = 1892$, and $p=3$ (right) for $\texttt{nDof} = 2053$. In both cases, the mesh is refined mainly for the primal singularities in $\Omega_1$ in the lower left corner.}
\label{figure:smoothSolutionMeshes}
\end{figure}

\begin{figure}
\includegraphics[width=0.48\textwidth]{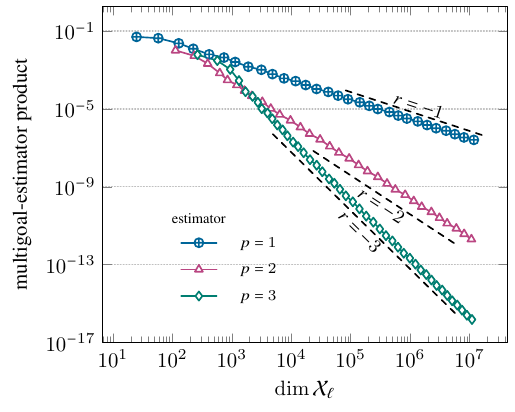}
\includegraphics[width=0.48\textwidth]{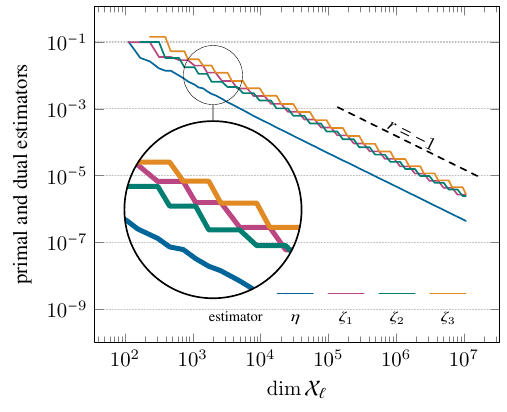}
\caption{Convergence plots for the experiment from section~\ref{experiment:smooth}: Left: Convergence of the multigoal-estimator product $\Delta_\ell = \eta_\ell \sum_{i=1}^3 \zeta_{i,\ell}$ from~\eqref{eq:quasi-error:multigoal} for polynomial degrees $p\in \{1,2,3\}$. Right: Convergence of the individual estimators $\eta$ as well as $\zeta_1$, $\zeta_2$, and $\zeta_3$ for $p=2$. The estimators on the right-hand side are colored in accordance with the initial mesh (middle panel in Figure~\ref{figure:smoothSolutionMeshes}).}
\label{figure:smoothSolutionConvergence}
\end{figure}

\begin{figure}
\includegraphics[width=0.48\textwidth]{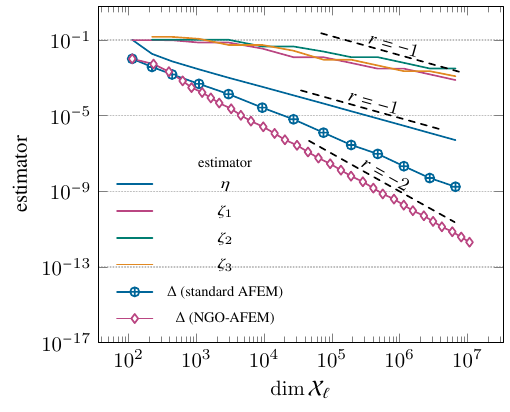}
\includegraphics[width=0.48\textwidth]{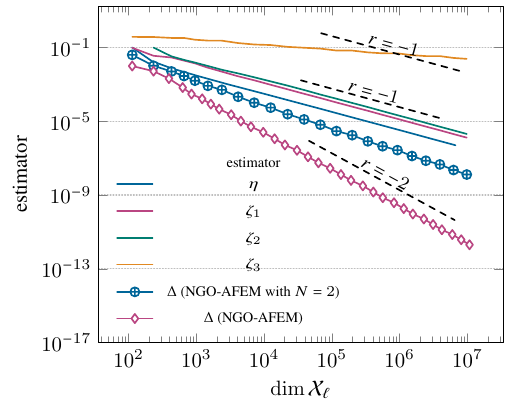}
\caption{Convergence plots for the experiment from Section~\ref{experiment:smooth} for alternative schemes with $p=2$. For further details, see subsection~\ref{experiment:smooth}. Left: Classical AFEM. Right: Standard NGO-AFEM with $N=2$, where goal $G_3$ is never active (while its estimator $\zeta_3$ is still computed), i.e., step~(\ref{algorithm:solve:dual}) in Algorithm~\ref{algorithm:Ngoal} solves only for $j \in \{1,2\}$. The multigoal-estimator product for NGO-AFEM with $N=3$ from Figure~\ref{figure:smoothSolutionConvergence} is added for reference.}
\label{figure:smoothSolutionConvergenceAFEM}
\end{figure}

\subsection{Experiment with \texorpdfstring{$\boldsymbol{N=8}$}{N=8} goals}\label{experiment:nonsmooth}

We consider quadratic FE spaces on the Z-shaped domain $\Omega = (-1,1)^2 \setminus \textup{conv}\{(0,0), (-1,0), (-1,-1)\} \subset \R^2$ with Dirichlet boundary conditions on the reentrant corner $\textup{conv}\{(-1,0), (0,0)\} \cup \textup{conv}\{(0,0), (-1,-1)\}$ and Neumann boundary conditions on the remaining part of the boundary $\partial \Omega$; see Figure~\ref{figure:nonsmoothSolutionMeshes} for the geometric setup with the Dirichlet boundary in red. The diffusion coefficient is given by $\boldsymbol{A} = 1$ and the right-hand sides are given by $f=0$ and $g_i = 0$ for $i=1, \ldots, 8$ as well as $\boldsymbol{f} = (-10,0)^\top \chi_{\Omega_1}$, $\boldsymbol{g}_i = (-10, 0)^\top \chi_{\Omega_{i}}$ for $\text{mod}(i, 3) = 0$, $\boldsymbol{g}_i = (1, 0)^\top \chi_{\Omega_{i}}$ for $\text{mod}(i, 3) = 1$, and $\boldsymbol{g}_i = (0, 100)^\top \chi_{\Omega_{i}}$ for $\textup{mod}(i, 3) = 2$, where $i \in \{1, \ldots, 8\}$. This ensures that all directions touch both boundary conditions. The adaptivity parameters read $\theta =0.3$, $\Cmark = 2$, and $\varrho_{\textup{irr}} = 0.1$.

Figure~\ref{figure:nonsmoothSolutionMeshes} (right) also displays a snapshot mesh of NGO-AFEM after several adaptive steps, where the mesh is refined towards the reentrant corner to resolve the primal singularity as well as towards all subdomains $\Omega_{\boldsymbol{g}_i}$ for $i=1, \ldots, 8$ to resolve the dual singularities.

\begin{figure}
\includegraphics[height=0.3\textwidth]{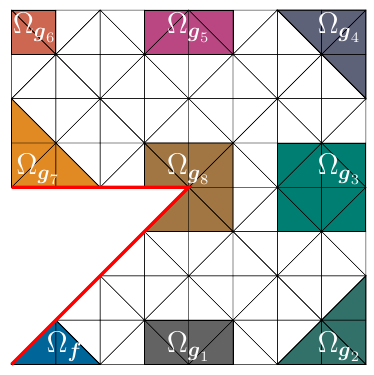}
\hfil
\includegraphics[height=0.30\textwidth]{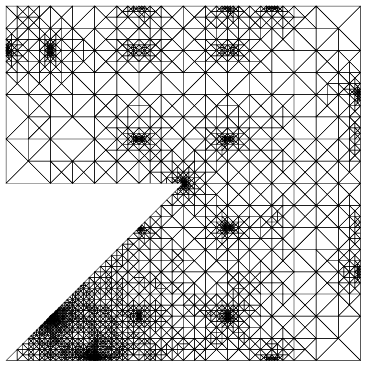}
\caption{Geometric setup of Experiment with nonsmooth solution from Section~\ref{experiment:nonsmooth}: Initial adaptive mesh on the Z-shaped domain $\Omega$ with reentrant corner at the origin. All subdomains $\Omega_{\boldsymbol{f}}$ and $\Omega_{\boldsymbol{g}_i}$ with $i=1, \ldots, 8$ are indicated. The boundary on the reentrant corner is highlighted in red, indicating homogeneous Dirichlet boundary conditions, while its complement features homogeneous Neumann boundary conditions.}
\label{figure:nonsmoothSolutionMeshes}
\end{figure}

The experiment compares four instances of Algorithm~\ref{algorithm:Ngoal} that differ in two aspects: 
\begin{enumerate}
    \item For irregular marking in Algorithm~\ref{algorithm:Ngoal}(\ref{algorithm:mark}), we either select exactly $\# \MM_{\ell-1}$ elements from the set $\MM_{\ell}^{uz}$ to form the set $\MM_\ell$ with marking from~\cite{fpz2016}, or none at all, i.e., $\MM_\ell = \emptyset$;
    \item \label{variant:two} We sort the error estimators $\zeta_{i,0}$ for $i=1, \ldots, N$ as proposed in Corollary~\ref{cor:multigoal:optimalRates}, or we do not sort them at all.
\end{enumerate}
Figure~\ref{figure:nonsmoothComparison} depicts all four cases. The $x$-axis is adjusted to display the cumulative degrees of freedom $\texttt{cumnDof} \coloneqq  \sum_{0 \le \ell} \texttt{nDof}_\ell$ and optimal convergence rates $\mathcal{O}(\texttt{cumnDof}^{-1})$ are observed in all cases over the cumulative degrees of freedom. 

We see that the preasymptotic regime is fairly long in all cases, but variant [$\# \MM_{\ell-1}$ marked, not sorted] (top left) reaches the asymptotic regime the earliest (see $\zeta_7$ for example). We also see that standard NGO-AFEM [$\# \MM_{\ell-1}$ marked, not sorted] (top left) is the only variant, where the primal estimator is not a lot smaller than all dual estimators. We note that the case of no irregular marking, i.e., the variant [$0$ marked, not sorted] (top right) and the variant [$0$ marked, sorted] (bottom right), the estimator $\eta_\ell$ indicates levels where no refinement is performed, while in the case of variant [$\# \MM_{\ell-1}$ marked, not sorted] (top left) and variant [$\# \MM_{\ell-1}$ marked, sorted] (bottom left), the estimator $\eta_\ell$ is wiggly but always reduced. 

\begin{figure}
    \centering
    \subfloat[$\# \MM_{\ell-1}$ marked, not sorted]{
        \includegraphics[width=0.475\textwidth]{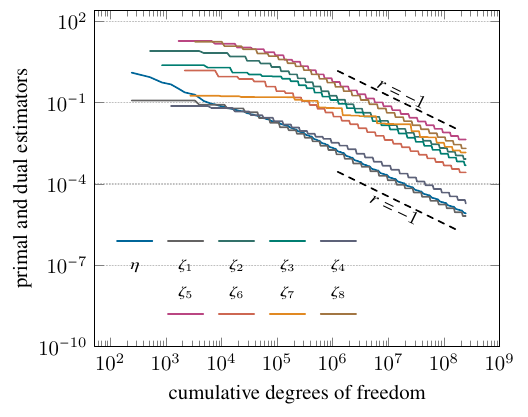}
        \label{figure:nonsmoothComparison:one}
    }
    \subfloat[$0$ marked, not sorted]{
        \includegraphics[width=0.475\textwidth]{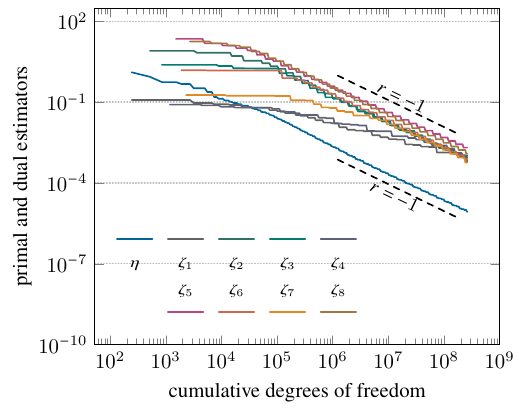}
        \label{figure:nonsmoothComparison:two}
    }

    \medskip
    \subfloat[$\# \MM_{\ell-1}$ marked, sorted]{
        \includegraphics[width=0.475\textwidth]{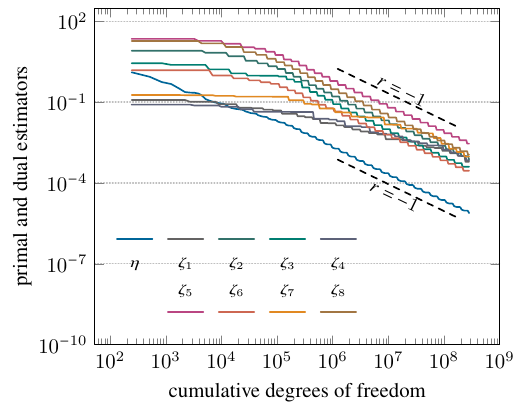}
        \label{figure:nonsmoothComparison:three}
    }
    \subfloat[$0$ marked, sorted]{
        \includegraphics[width=0.475\textwidth]{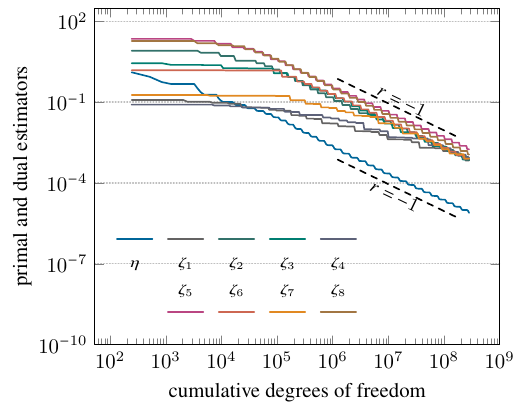}
        \label{figure:nonsmoothComparison:four}
    }

\caption{Convergence plots for the experiment from Section~\ref{experiment:nonsmooth} for the individual estimators $\eta$ as well as $\zeta_1, \ldots, \zeta_8$ for $p=2$. The estimators on the left-hand side are colored in accordance with the initial mesh in Figure~\ref{figure:nonsmoothSolutionMeshes}.}
\label{figure:nonsmoothComparison}
\end{figure}

Observing the multigoal-error estimator $\Delta_\ell$ from~\eqref{eq:quasi-error:multigoal} in Figure~\ref{figure:nonsmoothSolutionConvergenceProduct}, we see that all four variants perform equally well and reach optimal convergence rates $\mathcal{O}(\texttt{cumndof}^{-2})$, with a very small favor for the variants without irregular marking; see the zoom-in in Figure~\ref{figure:nonsmoothSolutionConvergenceProduct}.

\begin{figure}
\includegraphics[height=0.3\textwidth]{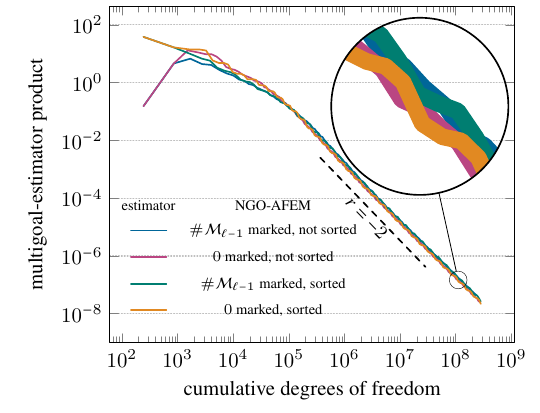}
\caption{Convergence plot for the experiment from Section~\ref{experiment:nonsmooth} for the multigoal-estimator product $\Delta_\ell = \eta_\ell \sum_{i=1}^8 \zeta_{i,\ell}$ from~\eqref{eq:multigoal:error} for all four variants explained in subsection~\ref{experiment:nonsmooth}.}
\label{figure:nonsmoothSolutionConvergenceProduct}
\end{figure}

%% file: 06_conclusion.tex
\section{Conclusion}\label{section:conclusion}
We have presented and analyzed a novel multigoal-oriented AFEM with $N$ goals (NGO-AFEM) for symmetric linear elliptic PDEs that can simultaneously deal with multiple linear goal functionals while only solving two discrete problems per adaptive step. The thorough convergence analysis verifies that the proposed algorithm guarantees R-linear convergence with optimal convergence rates. The gist of the new algorithm is a marking strategy that combines the D\"orfler marking for the primal and one active dual problem but also exercises a certain cardinality control if the currently active dual problem is not the most pressing one in terms of the size of the corresponding error estimator. 

To keep the presentation of this paper reasonable and focused on the new marking strategy, we have considered a symmetric linear elliptic PDE model problem and linear goal functionals. Extensions may revolve around the following aspects: 

First, one may consider nonsymmetric or even nonlinear PDEs as model problems. This would require to adapt the dual problems accordingly, but the crucial Pythagorean identity for linear convergence may be replaced by arguments from~\cite{feischl2022} that apply to general second-order linear elliptic PDEs or compactness arguments which guarantee a quasi-Pythagorean identity that applies also to certain nonlinear PDEs~\cite{ffp2014, bhp2017, bbimp2022}. Results on goal-oriented AFEM with \emph{single} goal within this problem class are, e.g.,~\cite{bip2020, bbps2025} but also (locally Lipschitz continuous) semilinear problems are possible~\cite{bbimp2022}.

Second, one may consider an algebraic solver to solve the discrete problems by means of an inner loop. The algorithmic interplay of adaptive mesh refinement and inexact algebraic solvers has been studied in~\cite{ghps2021, hpsv2021, bhimps2023} for standard AFEM and also in the goal-oriented setting with \emph{single} goal in~\cite{bip2020,dc2023, bgip2023, bbps2025}. We expect that the main ideas of our algorithm and its analysis can be transferred to this more general setting.

Third, also quadratic goals in the spirit of~\cite{bip2020} fall within the scope of potential extensions for future research. We note that the extension to nonlinear problems and quadratic goals seems to be more involved, since the primal and dual problems do not fully decouple as in the linear case. Already for a single goal, the goal estimate will be of the form $\eta_{\ell}^2 + \eta_\ell \zeta_\ell$ (instead of only $\eta_\ell \,\zeta_\ell$) and, thus, the marking strategy and its analysis need to be adapted accordingly.

Finally, the marking strategy from~\cite{bet2011} first combines the estimator quantities to a single estimator and then marks the mesh for refinement (whereas the approach from~\cite{ms2009, fpz2016} is of the form \emph{mark first, combine second}). This alternative marking strategy is a more global version of the so-called DWR strategy that numerically performs very well and thus is widely used in practice. It will be interesting to understand whether such a marking strategy can also be analyzed within the framework of an adapted \texttt{MARK} module as presented in this work.